\documentclass[12pt,reqno]{amsart} 
\usepackage[mathscr]{eucal} 
\usepackage{amsmath,amsfonts}
\usepackage{amssymb} 
\usepackage{algorithm}
\usepackage{algorithmic}
\usepackage{color}
\usepackage{graphicx}

\parskip=\smallskipamount 
 
\newtheorem{theorem}{Theorem}[section] 
\newtheorem{proposition}[theorem]{Proposition} 
\newtheorem{corollary}[theorem]{Corollary} 
\newtheorem{lemma}[theorem]{Lemma} 
 
\theoremstyle{definition} 
 
\newtheorem{example}[theorem]{Example}

\newtheorem{remark}[theorem]{Remark}

\makeatletter 
\@addtoreset{equation}{section} 
\makeatother

\newcommand{\CC}{{\mathbb C}} 
\newcommand{\FF}{{\mathbb F}}
\newcommand{\NN}{{\mathbb N}}

\newcommand{\RR}{{\mathbb R}} 
 
\newcommand{\cB}{{\mathcal B}}

\newcommand{\cF}{{\mathcal F}} 
 
\newcommand{\cH}{{\mathcal H}} 
\newcommand{\cI}{\mathcal{I}}
 
\newcommand{\cK}{{\mathcal K}}

\newcommand{\cN}{{\mathcal N}} 
 
\newcommand{\cR}{{\mathcal R}} 
\newcommand{\cS}{{\mathcal S}}

\newcommand{\cV}{{\mathcal V}} 
\newcommand{\cW}{{\mathcal W}} 
\newcommand{\cX}{\mathcal{X}}
\newcommand{\cY}{{\mathcal Y}}
\newcommand{\cZ}{{\mathcal Z}} 
\newcommand{\bC}{\boldsymbol{C}}
\newcommand{\bW}{\boldsymbol{W}}
\newcommand{\bY}{\boldsymbol{Y}}
\newcommand{\bH}{\boldsymbol{H}}

\newcommand{\bV}{\boldsymbol{V}}
\newcommand{\ba}{\boldsymbol{a}}

\newcommand{\bef}{\boldsymbol{f}}
\newcommand{\bu}{\boldsymbol{u}}
\newcommand{\bz}{\boldsymbol{z}}
\newcommand{\bx}{\boldsymbol{x}}
\newcommand{\bw}{\boldsymbol{w}}
\newcommand{\by}{\boldsymbol{y}}
\newcommand{\fK}{\mathfrak{K}}
\newcommand{\diag}{\operatorname{diag}}
\newcommand{\argmin}{\operatorname{argmin}}
\newcommand{\dom}{\operatorname{Dom}} 
 
\newcommand{\Ra}{\Rightarrow}

\newcommand{\ra}{\rightarrow} 
\newcommand{\ol}{\overline}

\let\phi=\varphi 
\newcommand{\iac}{\mathrm{i}}

\newcommand{\deriv}{\operatorname{D}}

\newcommand{\lin}{\operatorname{Lin}}

\newcommand{\norm}[1]{\left\lVert#1\right\rVert}
\newcommand\innerprod[2]{\left\langle #1, #2 \right\rangle}

\title[Localisation of Regularised and Multiview Learning]{Localisation of Regularised and Multiview Support Vector Machine Learning}
 
\author[A. Gheondea]{Aurelian Gheondea} 

\address{Institutul de Matematic\u a al Academiei Rom\^ane, 21 Calea Grivi\c tei  010702 
Bucure\c sti, Rom\^ania \emph{and}
Department of Mathematics, Bilkent University, 06800 Bilkent, Ankara, 
Turkey} 
\email{A.Gheondea@imar.ro \textrm{and} aurelian@fen.bilkent.edu.tr } 

\author[C. T\.\i lk\.\i]{Cankat T\.\i lk\.\i}
\address{Department of Mathematics and Division of Computational Modeling and Data Analytics, Virginia 
Polytechnic Institute and State University, Blacksburg Virginia, 24061 U.S.A.}
\email{cankat@vt.edu}
\begin{document}

\begin{abstract}
We prove some representer theorems for a localised version of a semisupervised,  
manifold regularised and multiview support vector machine learning 
problem introduced by H.Q.~Minh, L.~Bazzani, and V.~Murino,  \textit{Journal of Machine Learning 
Research}, \textbf{17}(2016) 1--72, that involves operator valued positive semidefinite kernels and their 
reproducing kernel Hilbert spaces. The results concern general cases when convex or nonconvex loss
functions and finite or infinite dimensional underlying Hilbert spaces are considered. 
We show that the general framework allows infinite dimensional Hilbert spaces and nonconvex loss 
functions for some special cases, in particular in case the loss functions are G\^ateaux 
differentiable. 
Detailed calculations are provided for the exponential least squares loss functions that 
lead to systems of partially nonlinear equations for which a particular different types of Newton's 
approximation methods based on the interior point method 
can be used. Some numerical experiments are performed on a toy model that illustrate
the tractability of the methods that we propose.
\end{abstract}

\keywords{operator valued reproducing kernel Hilbert spaces, regularised and multiview learning, 
support vector machine learning, loss functions,
representer theorem}

\subjclass[2010]{Primary 68T05; Secondary 46E22, 46G05}

\maketitle

\section{Introduction}

Representer theorems are of a special interest in machine learning due to the fact that 
they reduce the problem of finding a minimiser for the learning map 
to the vector space spanned by the kernel functions, or operators, at the labeled 
and unlabeled input data. For classical versions of representer theorem, we recommend 
the monographs  \cite{Schoelkopf} and \cite{SteinwartChristmann}. There is a large 
literature on generalised representer theorems but in this article we refer to the unifying 
framework in vector valued reproducing 
kernel Hilbert spaces for semisupervised, manifold regularised and multiview machine 
learning, as investigated by \cite{MBM} and the vast literature cited there. 

The article \cite{MBM} has remarkable contributions to the domain of representer 
theorems in support vector machine learning, firstly 
by unifying many variants of these theorems referring to semisupervised, regularised, 
manifold regularised,  multiview machine learning and then by considering underlying 
Hilbert spaces that are infinite dimensional. Recently, infinite dimensional Hilbert spaces 
in learning with kernels have been of interest, e.g.\ see \cite{Lambert}.
However, although the general representer theorem, Theorem 2 in \cite{MBM}, 
is stated for infinite dimensional spaces,
this turns out to be problematic, as we will see in Remark~\ref{r:inf}. Also, there is an
interest for applications to learning problems in which loss functions may not 
be convex, cf.\ \cite{ZMY}, or even indefinite, cf.\ \cite{Kwon}.

In this article we are concerned with questions triggered by the investigations in 
\cite{MBM} and \cite{ZMY},
such as: to which extent can one allow the underlying input spaces be infinite 
dimensional 
and to which extent nonconvex loss functions can be used in the learning process. In 
this respect, we
propose a localisation of the minimisation of the semisupervised regularised, manifold 
regularised, and multiview machine learning 
problem studied in \cite{MBM}, in the sense that the output spaces and the loss 
functions may be different for each labeled and unlabeled input point. For this 
approach we use a generalised version of vector valued reproducing kernel Hilbert 
spaces with bundles of spaces and operators. We think that the localised framework 
offers more flexibility to the learning 
problem and it is quite natural, especially when semisupervised multiview learning is 
considered, that the 
output spaces and the loss functions depend locally on the input points.

{ There are a few reasons that motivate the localised versions that we consider in this article. Firstly,
for some of the labeled input points, in the multivariable case, some of the components of the labels (properties) 
may be missing or some 
additional components of the labels may be necessary in order to allow reliable information. This 
means that the underlying vector spaces of labels may have different dimensions and hence, making the vector 
spaces of the labels depend on the input points solves this obstruction. Secondly, when the input set $X$ shows 
a certain homogeneity, the localised version may not be needed but, when this set is more heterogeneous, the 
localised version brings the necessary flexibility. To be more precise, let us imagine the following scenario. The 
input set $X$ is a finite union of sets $X_i$, $i=1,\ldots,N$, where each $X_i$ shows homogeneity. By 
homogeneity we mean that, for each $i$, the properties (labels) associated to $x\in X_i$ are of the same type, in 
particular they live in the same Hilbert space $\cY_i$. 
However,  for different $i\neq j$, where $i,j\in \{1,\ldots,N\}$, the properties (labels) of the points $x\in X_i$ when 
compared to the properties of the points $x\in X_j$ may be different, meaning that the Hilbert spaces of labels 
$\cY_i$ and $\cY_j$ should be different. For example, if the set $X$ is the collection of all the cells 
of a body, we can see $X$ as the union of all its organs $X_i$, for $i=1,\ldots,N$, and it is clear that due to 
special functions that different organs have in the body, the properties (labels) of cells in different organs are 
generally of different type. 
Moreover, the learning function, which is calculated in terms of loss functions, may require 
different loss functions for different points that have different types of labels. For example, in the scenario of the 
body viewed as the union of its organs, comparison of properties of different cells should be performed differently 
for cells belonging to different organs. This justifies the dependence of the loss functions of the input points as 
well. As a further research project, the localised version of the machine learning problem 
might be used for investigating the process of spreading a malignant tumor (metastasis) in different organs of a 
body, provided that real data will be available and appropriate mathematical models will be obtained.}

Following \cite{MBM}, the direction of multiview learning we consider in this work 
is coregularisation, 
see e.g.\ \cite{Brefeld}, \cite{SindhwaniRosenberg}, \cite{Rosenberg}, and \cite{Sun}. In 
this approach, different hypothesis spaces are used to construct target functions based 
on different views of the input data, such as different features or modalities, and a 
data dependent regularisation term is used to enforce consistency of output values 
from different views of the same input example. The resulting target functions, each 
corresponding to one view, are then naturally combined together in a certain 
fashion to give the final solution.

The direction of semisupervised learning we follow here is manifold regularisation, 
cf.\ \cite{Belkin}, \cite{Brouard}, and \cite{MinhSindhwani}, which attempts to learn the 
geometry of the input space by exploiting the given unlabeled data. The latter two 
papers are recent generalisations of the original scalar version of manifold regularisation 
of \cite{Belkin} to the vector valued setting. In \cite{Brouard}, a vector valued version of 
the graph Laplacian $L$ is used while in \cite{MinhSindhwani} $L$ is a general 
symmetric, positive operator, including the graph Laplacian. The vector valued setting 
allows one to capture possible dependencies between output variables by the use of, for 
example, an output graph Laplacian. For a comprehensive discussion on 
semisupervised learning and a thorough comparison with supervised and unsupervised 
learning, see the collection \cite{Chapelle}.

Because reproducing kernel 
Hilbert spaces make an essential ingredient in these representer type theorems, some historical 
considerations are in order.
The classical article of  \cite{Aronszajn} provides an abstract formalisation of scalar 
reproducing kernel Hilbert spaces of many previous investigations and applications 
related to spaces of analytic 
functions, partial differential equations, and harmonic analysis, as of  \cite{Mercer}, 
\cite{Bergman}, \cite{Moore},  \cite{Bochner}, and  \cite{Godement}. From the point of 
view of probability theory,  \cite{Kolmogorov} investigated
linearisations, or feature spaces, associated to scalar positive semidefinite kernels. 
 An equivalent formulation of reproducing kernel Hilbert spaces by
Hilbert spaces continuously embeded in a quasicomplete locally convex space was 
investigated by \cite{Schwartz} while a more group theoretical approach was performed in \cite{Krein1} and 
\cite{Krein2}. Then, scalar valued 
reproducing kernel Hilbert spaces found many applications in machine learning, see \cite{SteinwartChristmann}
and \cite{Schoelkopf} for a comprehensive list of literature in this direction.

Motivated by problems in operator theory and operator algebras, operator valued positive semidefinite kernels 
and either their linearisations (Kolmogorov decompositions) or their reproducing kernel Hilbert spaces
have been considered by  \cite{BSzNagy}, \cite{Pedrick},  \cite{ParthasarathySchmidt}, 
\cite{EvansLewis}, and   \cite{Constantinescu}, to cite a few. Investigations of operator 
valued Hermitian kernels that yield Krein spaces have been performed by
\cite{ConstantinescuGheondea1}, \cite{ConstantinescuGheondea2}. 
More general operator valued positive semidefinite kernels that yield reproducing kernel 
VH-spaces have been considered by \cite{Gheondea}, \cite{AyGh2015}, 
\cite{AyGh2017}, and \cite{AyGh2019}. { To make things more precise, recall that a vector space
$\cS$ endowed with a map $\cS\ni s\mapsto s^*\in\cS$ which is conjugate linear and involutive, 
in the sense that $(s^*)^*=s$ for all $s\in \cS$, is called a \emph{$*$-vector space}. If, in addition, a convex cone 
$\cS^+$ on $\cS$ is specified such that for any $s\in \cS^+$ we have $s=s^*$, we call $\cS$ an ordered 
$*$-space and, when a complete locally convex topology on $\cS$ is specified 
and which is related in a certain fashion with 
the convex cone $\cS^+$, one calls $\cS$ an \emph{admissible space}. The concept of admissible space is a 
generalisation of the concept of $C^*$-algebra which is a mainstream domain in functional analysis, partly 
due to its strong connections with quantum theories, e.g.\ see \cite{Davidson} and \cite{Blackadar}.
Admissible spaces $\cS$ are then used to 
consider gramians $[\cdot,\cdot]\colon \cX\times\cS\ra \cS$, for some vector space $\cX$, 
which are vector valued inner products and that induce a certain topology on $\cX$. If this topology is complete 
then we call $\cX$ a VH-space (Vector Hilbert space). The concept of VH-space is a generalisation of the 
concept of  Hilbert modules over $C^*$-algebras, e.g.\ see \cite{Lance}, but it appeared independently and 
related to problems in probability theory, see \cite{Loynes1}, \cite{Loynes2}, and \cite{Loynes3}.
In case we have a positive semidefinite kernel 
$K\colon X\times X\ra \cS$, for some admissible space $\cS$, one can define a reproducing kernel VH-space 
as a generalisation of the classical reproducing kernel Hilbert space. 
For details and many examples see, for example, 
\cite{AyGh2017}.} 

During the last twenty years, operator valued 
positive semidefinite kernels and their reproducing kernel Hilbert spaces and feature 
spaces (linearisations, or Kolmogorov decompositions) became of 
interest in the theory of machine learning, see 
\cite{Micchelli}, \cite{Carmelli}, \cite{Caponnetto}, \cite{MinhSindhwani}, \cite{Kadri},
but the investigations have been somehow started from scratch, apparently unaware of 
the previous works.  { More recently, in a sequence of articles, \cite{Hashimoto21}, \cite{Hashimoto22},
\cite{Hashimoto23a}, \cite{Hashimoto23b}, and \cite{Hashimoto24}, it is shown that, using the 
$C^*$-algebra-valued gramians (vector valued inner products), 
one can learn function- and operator-valued maps, one can design positive definite kernels for structured data 
using the noncommutative product, one can use the norm of the $^*$-algebra to alleviate the dependency of 
generalisation error bound on the output dimension using the generalisation of kernel mean embedding by 
means of $C^*$-algebras, one can analyse positive operator valued measures 
and spectral measures, one can continuously combine multiple models and use the tools for functions which 
can be applied to ensemble, multitask, and meta-learning (the noncommutative product structures in 
$C^*$-algebras induce interactions among models), and one can construct group equivariant neural networks 
using the products in group $C^*$-algebras.}

We briefly describe the contents of this article. In order to introduce the localised version 
of the semisupervised, regularised, manifold regularised, and multiview learning 
problem, in Subsection~\ref{ss:ovk} we firstly consider operator valued 
positive semidefinite kernels for which the entries are localised by a bundle of Hilbert 
spaces. For these kernels, we 
show how their reproducing kernel Hilbert spaces are constructed, their relation to 
the linearisations (Kolmogorov 
decompositions, feature spaces) and basic properties. Although we have been inspired 
by the approach in \cite{ConstantinescuGheondea1}, 
we provide in the appendices proofs for all the statements we make since in this form 
they cannot be found elsewhere. In \cite{Pedrick}, reproducing kernel Hilbert spaces are 
considered for kernels over bundles of locally convex spaces and this might be a very 
interesting research problem in case applications to machine learning theory might be 
found.

Then, in Subsection~\ref{ss:grmlp} we present the localised version of the 
semisupervised, regularised, manifold regularised, and multiview  machine
learning problem inspired by \cite{MBM}. Under rather general assumptions, we prove in 
Subsection~\ref{ss:ovk} the 
representer theorem for general loss functions but assuming the input spaces at all input 
points be finite dimensional. The finite dimensionality assumption can be relaxed to 
the condition that the span of the input kernel operators at the input points is closed, as 
Proposition~\ref{p:minimizer} shows. In this section, most of the underlying Hilbert 
spaces may be complex and only in some special cases we have to impose the 
condition that they are real.

Further on, in Subsection~\ref{ss:rtldlf} we consider real underlying Hilbert spaces and 
loss functions that are G\^ateaux 
differentiable and then, in Theorem~\ref{t:gatmin} we show that, under this additional 
assumption, a general representer theorem can be obtained for infinite dimensional 
input spaces. In Subsection~\ref{ss:lslf} we work out the details for the loss 
functions defined by the least squares which leads to a linear problem in terms of the 
unkown coefficients, as in \cite{MBM}, while
in Subsection~\ref{ss:elslf} we work out the details for the loss functions 
defined by the exponential least squares which lead to a mixed problem, linear and nonlinear. 
{Finally, in
Subsection~\ref{ss:aels} we tackle algorithms to obtain approximations of the solutions for the latter machine 
learning problem by damped Newton approximation methods, as presented in \cite{MonteiroPang}, 
\cite{Byrd99}, \cite{Byrd00}, \cite{Waltz}. Since in the framework of systems of nonlinear 
equations solutions are generally hot topics in current research, much work 
remains to be done. Some numerical 
experiments are performed on a toy model for the latter case and the algorithm is tested. We show the 
robustness of the method on this toy model.}\medskip

\textbf{Acknowledgements}: {The authors thank Naci Saldi and Sorin Costiner for very helpful discussions on 
available algorithms to solve nonlinear systems of equations and optimisation problems.

\section{A General Semisupervised Regularised and Multiview
Machine Learning Problem}\label{s:rmp}

\subsection{Operator Valued Kernels.}\label{ss:ovk}
{ The aim of this subsection is to introduce kernels that take values in bundles of bounded linear
operators on different Hilbert spaces and which will make the theoretical foundations for the localised machine 
learning problem. Actually, we prove that passing from nonlocalised to localised machine learning problem has 
the advantage of being very flexible and allowing 
a very large nonhomogeneity of the input data while not bringing
new obstructions. In particular, this shows that, with minimal changes, similar results and algorithms can be 
obtained for localised machine learning problems.}

Let $X$ be a nonempty set and $\bH=\{\cH_x\}_{x\in X}$ a bundle of 
Hilbert spaces over the field $\FF$, that is either $\RR$ or $\CC$, 
with base $X$, that is, $\cH_x$ is a Hilbert space over 
$\FF$, for any $x\in X$. In order to avoid confusion, let us note that the Hilbert spaces 
$\cH_x$ are actually tagged Hilbert spaces, meaning that for $x\neq y$ the spaces 
$\cH_x$ and $\cH_y$ are disjoint. An \emph{$\bH$-operator valued 
kernel} $K$ is a mapping defined on $X\times X$ such that 
$K(y,x)\in\cB(\cH_x,\cH_y)$ for all $x,y\in X$. Here and throughout this article, if $\cH$ 
and $\cK$ are two Hilbert spaces over the same filed $\FF$, we denote by 
$\cB(\cH,\cK)$ the Banach space of all linear and bounded operators 
$T\colon \cH\ra\cK$, endowed with the operator norm. We denote by $\fK(X;\bH)$ the 
collection of all $\bH$-operator
valued kernels on $X$ and it is clear that $\fK(X;\bH)$ is a vector space over $\FF$.

Given $K\in\fK(X;\bH)$, the adjoint kernel $K^*$ is defined 
by $K^*(x,y)=K(y,x)^*$ for all $x,y\in X$. Clearly $K^*\in\fK(X;\bH)$. The kernel $K$ is called \emph{Hermitian} or 
\emph{symmetric} if 
$K=K^*$.  If $\FF=\CC$ then any kernel $K$ is a linear combination of two Hermitian kernels, more precisely, letting
\begin{equation}\label{e:rem}\Re(K):=(K+K^*)/2,\quad \Im(K):=(K-K^*)/2\iac,\end{equation}
we have
\begin{equation}\label{e:remek}
K=\Re(K)+\iac \Im(K).
\end{equation}
It is easy to see that $K\mapsto K^*$ is an involution, that is, it is conjugate linear and involutive.
In this way, $\fK(X;\bH)$ is a $*$-vector space, { that is, $\fK(X;\bH)$ is a vector space endowed with an involution, e.g.\ see \cite{AyGh2015}.}

Let $\cF(X;\bH)$ be the vector space over $\FF$ of all cross-sections 
$f\colon X\rightarrow \bigcup_{x\in X}\cH_x$, that is, $f(x)\in \cH_x$ for all $x\in X$. 
Addition and multiplication with scalars in $\cF(X;\bH)$ are defined pointwise.
Equivalently, $\cF(X;\bH)$ can be naturally identified with the vector space $\prod\limits_ {x\in X} \cH_x$ and 
then any cross-section $f\in \prod\limits_ {x\in X} \cH_x$ 
can be written $f=\{f_x\}_{x\in X}$. For each $x\in X$ and $h\in\cH_x$ we consider the cross-section 
$\widehat h\in\cF(X;\bH)$, defined by
\begin{equation}\label{e:deltax}
\widehat h(y)=\begin{cases}  h,& \mbox{ if }y=x,\\ 0_{\cH_y},& \mbox{ otherwise.}
\end{cases}
\end{equation}
In particular, since $\bH$ consists of tagged Hilbert spaces this means that if either 
$h\in\cH_x$ or $l\in\cH_y$ are not null and $x\neq y$ then $\widehat h\neq \widehat l$. 
However, if $h=0_{\cH_x}$ and $l=0_{\cH_y}$ with $x\neq y$, 
then $\widehat h=\widehat l=0_{\cF(X;\bH)}$.
Also, for any $f\in\cF(X;\bH)$ we have
\begin{equation}\label{e:fesum}
f=\sum_{x\in X} \widehat f_x,
\end{equation}
where $f_x:=f(x)$ for all $x\in X$. Clearly, for each $y\in X$ the sum 
$\sum_{x\in X} \widehat f_x(y)$ has at most one nonnull term and hence the sum in 
\eqref{e:fesum} converges pointwise.

Let $\cF_0(X;\bH)$ be the vector subspace consisting of all $f\in\cF(X;\bH)$ with 
finite support. Clearly, any cross-section of type $\widehat h$, for some $h\in\cH_x$, belongs to $\cF_0(X;\bH)$ and, for 
any $f\in\cF_0(X;\bH)$ there exists uniquely distinct elements $x_1,\ldots,x_n\in X$ and 
$h_i\in\cH_{x_i}$, $i=1,\ldots,n$, such that 
\begin{equation*} f=\sum_{i=1}^n \widehat h_i.
\end{equation*}
 An inner product $\langle\cdot,\cdot\rangle_0\colon \cF_0(X;\bH)\times 
\cF_0(X;\bH)\ra \FF$ can be defined by
\begin{equation}\label{e:ipzero} \langle f,g\rangle_0=\sum_{x\in X}\langle f(x),g(x)
\rangle_{\cH_x},
\quad f,g\in \cF_0(X;\bH).
\end{equation}
In addition, 
let us observe that the sum in \eqref{e:ipzero} makes sense in the more general case when $f,g\in\cF(X;\bH)$ 
and at least one of $f$ or $g$ has finite support, the other can be arbitrary.  

Associated to the kernel $K\in\fK(X;\bH)$ there is a 
sesquilinear form $\langle \cdot,\cdot\rangle_K\colon \cF_0(X;\bH)\times 
\cF_0(X;\bH)\rightarrow \FF$ defined by
\begin{equation}\label{e:ipk} \langle f,g\rangle_K=\sum_{x,y\in X}\langle K(y,x) 
f(x),g(y)\rangle_{\cH_y},\quad f,g\in \cF_0(X;\bH),
\end{equation}
that is, $\langle\cdot,\cdot\rangle_K$ is linear in the first variable 
and conjugate linear in the second variable. Also, the sesquilinear form $\langle\cdot,\cdot
\rangle_K$ is Hermitian, that is, $\langle f,g\rangle_K=\overline{\langle g,f\rangle_K}$ for all $f,g\in \fK(X;\bH)$, 
if and only if the kernel $K$ is Hermitian. 

A \emph{convolution operator} $C_K\colon \cF_0(X;\bH)\ra \cF(X;\bH)$ can be defined by
\begin{equation}\label{e:conv} (C_Kf)(y)=\sum_{x\in X} K(y,x)f(x),\quad f\in
\cF_0(X;\bH),\ y\in X.
\end{equation}
Clearly $C_K$ is a linear operator and, with notation as in \eqref{e:ipzero} 
and \eqref{e:ipk} we have
\begin{equation}\label{e:convzero} \langle C_Kf,g\rangle_0=\langle f,g\rangle_K,
\quad f,g\in \cF_0(X;\bH).
\end{equation}

By definition, the kernel $K$ is \emph{positive semidefinite} if the sesquilinear form 
$\langle\cdot,\cdot\rangle_K$ 
is nonnegative, that is, if $\langle f,f\rangle_K\geq 0$ for all 
$f\in\cF_0(X;\bH)$, equivalently, if for all $n\in \NN$,
all $x_1,\ldots, x_n\in X$, and all $h_1\in \cH_{x_1},\ldots, h_n\in \cH_{x_n}$, 
we have
\begin{equation}\label{e:psk} \sum_{i,j=1}^n \langle K(x_j,x_i)h_i,h_j
\rangle_{\cH_{x_j}} \geq 0.
\end{equation}
An equivalent way of expressing \eqref{e:psk} is to say that the operator block 
matrix $[K(x_j,x_i)]_{i,j=1}^n$, when 
viewed as a bounded linear operator acting in the orthogonal 
direct Hilbert sum $\cH_{x_1}\oplus \cdots\oplus 
\cH_{x_n}$, is a positive semidefinite operator. On the other hand, the kernel $K$ is 
positive semidefinite if and only if the convolution operator, as defined in 
\eqref{e:conv}, is positive semidefinite when viewed 
as an operator on the inner product space 
$(\cF_0(X;\bH);\langle\cdot,\cdot\rangle_0)$, more precisely,
\begin{equation}
\langle C_Kf,f\rangle_0\geq 0,\quad f\in\cF_0(X;\bH).
\end{equation}
It is easy to see that, if $\FF=\CC$ then, if 
the kernel $K$ is positive semidefinite then it is Hermitian. If $\FF=\RR$ then this is not true and hence, for this case 
we confine to those positive semidefinite kernels that are Hermitian, more precisely, in this case, in addition to the 
property \eqref{e:psk}, by a positive semidefinite kernel we implicitly understand that it is Hermitian as well.
The collection of all
positive semidefinite $\bH$-operator valued kernels on $X$ is denoted by $\fK^+(X;\bH)$ and it is easy
to see that $\fK^+(X;\bH)$ is a strict convex cone of the $*$-vector space $\fK(X;\bH)$.

Given an arbitrary bundle of Hilbert 
spaces $\bH=\{\cH_x\}_{x\in X}$ and an $\bH$-operator valued Kernel $k$, a \emph{Hilbert space
linearisation}, or a 
\emph{Kolmogorov decomposition}, or a \emph{feature pair} of $K$ is, by definition, a pair $(\cK;V)$ subject to
the following conditions.
\begin{itemize}
\item[(kd1)] $\cK$ is a Hilbert space over $\FF$.
\item[(kd2)] $V=\{V(x)\}_{x\in X}$ is an 
operator bundle such that $V(x)\in \cB(\cH_x,\cK)$ for all $x\in X$.
\item[(kd3)] $K(x,y)=V(x)^* V(y)$ for all $x,y\in X$.
\end{itemize}
The linearisation $(\cK;V)$ is called \emph{minimal} if
\begin{itemize}
\item[(kd4)] $\cK$ is the closed span of $\{V(x)\cH_x\mid x\in X\}$.
\end{itemize}

The following theorem is a general version of some classical results, e.g.\ 
\cite{Kolmogorov}, \cite{ParthasarathySchmidt}, \cite{EvansLewis}. This is also a special
case of \cite{ConstantinescuGheondea1}. We include its proof in the Appendix A.

\begin{theorem}\label{t:kolmogorov} 
Given an arbitrary bundle of Hilbert 
spaces $\bH=\{\cH_x\}_{x\in X}$ and an $\bH$-operator valued kernel $K$, the 
following assertions are equivalent.
\begin{itemize}
\item[\emph{(a)}] $K$ is positive semidefinite.
\item[\emph{(b)}] $K$ has a Hilbert space linearisation.
\end{itemize}
In addition, if $K$ is positive semidefinite then a minimal Hilbert space linearisation 
$(\cK;V)$ exists and it is unique, modulo unitary equivalence, that 
is, for any other minimal Hilbert space linearisation $(\cK';V')$ of $K$ there exists
a unitary operator $U\colon \cK'\ra \cK$ such that $V(x)=UV'(x)$, for all $x\in X$.
\end{theorem}

Due to the uniqueness part in the previous theorem, for any $K\in\fK^+(X;\bH)$, we denote by $(\cK_K;V_K)$
the minimal Hilbert space linearisation of $K$, as constructed during the proof of the implication (a)$\Ra$(b).

Let $X$ be a nonempty set and 
$\bH=\{\cH_x\}_{x\in X}$ a bundle of Hilbert spaces over $\FF$.
Given an $\bH$-operator valued kernel $K$, 
a  \emph{reproducing kernel Hilbert space} associated to $K$ is, by definition, a Hilbert space 
$\cR\subseteq \cF(X;\bH)$ subject to the following conditions.
\begin{itemize}
\item[(rk1)] $\cR$ is a subspace of $\cF(X;\bH)$, with all induced algebraic operations.
\item[(rk2)] For all $x\in X$ and $h\in\cH_x$, the cross-section
$K_xh:=K(\cdot,x)h$ belongs to $\cR$. 
\item[(rk3)] For all $f\in\cR$ we have $\langle f(x),h\rangle_{\cH_x}
=\langle f,K_x h\rangle_\cR$, for all $x\in X$ and all $h\in\cH_x$.
\end{itemize}
Consequently, the following minimality condition holds as well:
\begin{itemize} 
\item[(rk4)] The span of $\{K_x h\mid x\in X,\ h\in\cH_x\}$ is dense in $\cR$.
\end{itemize}

Also, it is worth mentioning that by (rk2), for each $x\in X$, we actually have a bounded linear operator 
$K_x\colon \cH_x\ra\cR$ defined by $K_xh:=K(\cdot,x)h$, for all $h\in\cH_x$. This operator is bounded, as proven in 
\eqref{e:kix}. The following result is a generalisation of Moore-Aronszajn Theorem,
\cite{Moore}, \cite{Aronszajn},  \cite{Micchelli}, \cite{Carmelli}, \cite{Caponnetto},
\cite{MinhSindhwani}, \cite{Kadri}. Also, it is a special case of \cite{ConstantinescuGheondea1} and it is proven in the Appendix B. In Appendix C
we present a more direct construction of the reproducing kernel Hilbert space induced 
by an operator valued positive semidefinite kernel.

\begin{theorem}\label{t:rkh} Given an arbitrary bundle of Hilbert 
spaces $\bH=\{\cH_x\}_{x\in X}$ and an $\bH$-operator valued kernel $K$, the 
following assertions are equivalent.
\begin{itemize}
\item[\emph{(a)}] $K$ is positive semidefinite.
\item[\emph{(b)}] There exists a reproducing kernel Hilbert space $\cR$ having $K$
its reproducing kernel.
\end{itemize}
In addition, the reproducing kernel Hilbert space $\cR$ is uniquely determined by 
its reproducing kernel $K$.
\end{theorem}

\begin{remark}\label{r:kolrk} {There is a 
natural bijective transformation between the unitary equivalency class of minimal
linearisations $(\cK;V)$ of $K$ and the reproducing kernel Hilbert
space $\cR(K)$. The transformation from a minimal linearisations 
$(\cK;V)$ to the reproducing kernel Hilbert space $\cR(K)$ is described during the
proof of the implication (a)$\Ra$(b) of Theorem~\ref{t:rkh}, see Appendix B. In the following we
describe the inverse of this transformation.}

{Let $(\cR;\langle\cdot,\cdot\rangle_\cR)$ be a 
reproducing kernel Hilbert space with reproducing kernel $K$. 
We define the operator bundle $V=\{V(x)\}_{x\in X}$ by
\begin{equation}\label{e:v} V(x)h=K_xh,\quad x\in X,\ h\in\cH_x,
\end{equation} and remark that $V(x)\colon\cH_x\ra\cR$ for all $x\in X$. 
By means of the reproducing property (rk3) of the kernel $K$, we have
\begin{equation*} \langle V(x)h,V(x)h\rangle_\cR=\langle K_xh,K_xh\rangle_\cR
=\langle K(x,x)h,h\rangle_{\cH_x}\leq \|K(x,x)\| \|h\|^2_{\cH_x},\quad x\in X,\ 
h\in\cH_x,\end{equation*} hence $V(x)\in\cB(\cH_x,\cR)$. Also, using once more
the reproducing property (rk3) of $K$, it follows that, for all $x,y\in X$, $h\in\cH_x$,
and $g\in\cH_y$, we have
\begin{equation*} \langle V(y)^*V(x)h,g\rangle_{\cH_y}=\langle V(x)h,
V(y)g\rangle_\cR=\langle K_xh,K_yg\rangle_\cR=\langle K(y,x)h,g\rangle_{\cH_y}.
\end{equation*}
Therefore, $K(y,x)=V(y)^*V(x)$ for all $x,y\in X$ and hence, $(\cR;V)$ is a 
linearistion of $K$. In addition, using the minimality property (rk3),
it is easy to see that the linearisation $(\cR;V)$ is minimal as well.}
\end{remark}

One of the most important property of a reproducing kernel Hilbert space consists in the fact that, as a 
function space, its topology makes continuous all evaluation operators, see the proof in the Appendix D.

\begin{theorem}\label{t:evop} With notation as before, let $\cH$ be a Hilbert space in the vector space $\cF(X;\bH)$.
The following assertions are equivalent.
\begin{itemize}
\item[\emph{(a)}] $\cH$ is a reproducing kernel space of $\bH$-valued maps on $X$.
\item[\emph{(b)}] For any $x\in X$ the linear operator $\cH\ni f\mapsto f(x)\in \cH_x$ is 
bounded.
\end{itemize}
\end{theorem}

In connection to the previous theorem it is worth mentioning that, 
for a reproducing kernel Hilbert space $\cH_K\subseteq \cF(X;\bH)$ and
for arbitrary $x\in X$, the evaluation operator $\cH_K\ni f\mapsto f(x)\in \cH_x$ coincides with 
$K_x^*\colon \cH_K\ra \cH_x$, where $K_x\colon \cH_x\ra \cH_K$ is the bounded 
operator, see the axiom (rk2), defined by
$K_xh:=K(\cdot,x)h$, for all $h\in \cH_x$.

\subsection{Localisation of Semisupervised, Regularised, and Multiview Learning.}\label{ss:grmlp}

Let $X$ be a nonempty set and $\textbf{W}=\{\cW_x\}_{x\in X}$ be a bundle of Hilbert 
spaces on $X$. In this section,   
it is not important whether the Hilbert spaces are complex or real, hence all Hilbert 
spaces are considered to be 
over the field $\FF$, that is either $\CC$ or $\RR$. There is a difference between the 
complex and the real case 
consisting in the fact that in the latter case, for positive semidefiniteness we assume 
also the symmetry, or Hermitian, 
property, while in the complex case, the symmetry property is a consequence of the positive semidefiniteness.
 If $K$ is a 
positive semidefinite $\textbf{W}$-operator valued kernel, we let $\cH_K$ be its reproducing kernel Hilbert space, as 
in the previous subsection. Also, let $\bY=\{\cY_x\}_{x\in X}$ be a bundle of Hilbert spaces.

For $l,u\in \NN$, consider input distinct points $x_1,\ldots,x_{l+u}\in X$. Here $x_1,\ldots,x_l$ are the labeled 
input points while $x_{l+1},\ldots,x_{l+u}$ are the unlabeled input points. More precisely, there are 
given $y_1,\ldots,y_l$ output points, such that $y_j\in \cY_{x_j}$ for all $j=1,\ldots,l$. 
Then, for the general data let
\begin{equation*} \bx:=(x_j)_{j=1}^{l+u},\quad \by:=(y_j)_{j=1}^l,\quad
\bz:= \bigl((x_j)_{j=l+1}^{l+u},(y_j)_{j=1}^l\bigr).
\end{equation*}
The input points $x_1,\ldots,x_{l+u}$ are randomly 
selected with respect to an unknown probability and then, depending on the concrete 
problem, the labels $y_1,\ldots,y_l$ are produced in a certain way.

Let $\bW^{l+u}$ denote the Hilbert space
\begin{equation}\label{e:welu}
\bW^{l+u}=\bigoplus_{j=1}^{l+u} \cW_{x_j}.
\end{equation}
For $f\in \cH_K$ let
\begin{equation}\label{e:bef}
\bef:=(f(x_1),\ldots,f(x_{l+u}))\in \bW^{l+u}. 
\end{equation}

Also, there is given a (Hermitian, if $\FF=\RR$) positive semidefinite operator $M\in\cB(\bW^{l+u})$
represented as an operator block $(l+u)\times (l+u)$-matrix $M=[M_{j,k}]$, with $M_{j,k}\in \cB(\cW_{x_k},\cW_{x_j})$ 
for all $j,k=1,\ldots,l+u$.
Let $\bV=\{V_x\}_{x\in X}$ be a bundle of maps, loss functions, 
where $V_x\colon \cY_x\times \cY_x\ra \RR$ is a function, 
for all $x\in X$. Also, $\bC=\{C_x\}_{x\in X}$ is a bundle of bounded linear operators, where 
$C_x\colon \cW_x\ra \cY_x$ for all $x\in X$. The general minimisation problem is
\begin{equation}\label{e:fezga}
f_{\bz,\gamma}=\argmin\limits_{f\in \cH_K} \frac{1}{l}  \sum_{j=1}^l V_{x_j} (y_j,C_{x_j}f(x_j)) 
+ \gamma_A \|f\|_{\cH_K}^2+ \gamma_I \langle \bef,M\bef\rangle_{\bW^{l+u}},
\end{equation}
where $\gamma=(\gamma_A,\gamma_I)$ and $\gamma_A>0$ and $\gamma_I\geq 0$ 
are the regularisation parameters.

The optimisation problem \eqref{e:fezga} is a localised version of the general vector 
valued reproducing kernel Hilbert space for semisupervised, 
regularised, manifold regularised and multiview learning as in \cite{MBM}.
It is also useful to introduce the map to be minimised
\begin{align}\nonumber
\cI(f) & := \frac{1}{l}  \sum_{j=1}^l V_{x_j} (y_j,C_{x_j}f(x_j)) 
+ \gamma_A \|f\|_{\cH_K}^2+ \gamma_I \langle \bef,M\bef\rangle_{\bW^{l+u}}\\
\intertext{and, since $f(x)=K_x^*f$ for all $f\in \cH_K$ and all $x\in X$, it equals}
& = \frac{1}{l}  \sum_{j=1}^l V_{x_j} (y_j,C_{x_j}K_{x_j}^*f) 
+ \gamma_A \|f\|_{\cH_K}^2+ \gamma_I \langle \bef,M\bef\rangle_{\bW^{l+u}},\label{e:ielef}
\end{align}

In the following we explain the terms in the minimising map \eqref{e:ielef} and their 
significance from the point of view of machine learning. Firstly, why labeled and 
unlabeled data? Traditionally, in machine learning there are three fundamental 
approaches: supervised learning, unsupervised learning, and reinforcement learning, 
but the last one is out of our concern. We firstly recall the meaning and limitations of the 
first two approaches. \emph{Supervised learning} means that the training of the machine 
learning model is using exclusively 
labeled dataset. The input points and the 
labels are selected according to a probability that is usually unknown or the input points 
are selected according to an unknown probability and then the labels are produced in a 
certain fashion. Basically, this means that a label is a description showing a 
model, what it is expected to predict. But supervised learning has some limitations since 
this process is: \emph{slow}, because it requires human experts to either manually label 
training examples one by one or carefully supervise the procedure, and \emph{costly}, 
because, in order to obtain reliable results, a model should be trained on 
the large volumes of labeled data to provide accurate predictions.
\emph{Unsupervised learning} is that approach when a model tries to find hidden 
patterns, differences, and similarities in unlabeled data by itself, without human 
supervision. Most of the time, in this approach, data points are grouped into clusters 
based on similarities.
But, while unsupervised learning is a cheaper way to perform training tasks, it has other 
limitations: it has a \emph{limited area of applications}, mostly for clustering purposes,
and provides \emph{less accurate results}. 

\emph{Semisupervised learning} 
combines supervised learning and unsupervised learning 
techniques to solve some important issues: we train an initial model on a few labeled 
samples and then iteratively apply it to a greater number of unlabeled data.
Unlike unsupervised learning, semisupervised learning works for a larger variety of 
problems: classification, regression, clustering, and association.
Unlike supervised learning, the method uses small amounts of labeled data but large 
amounts of unlabeled data, with the advantage that it reduces the costs on human 
work and the data preparation time, while the accuracy of results is not altered. Of 
course, some other issues show up: the unlabeled points should show certain 
consistency and for this some regularisation techniques are needed.
A comprehensive discussion on this subject can be found in the collection 
\cite{Chapelle}.

Secondly, the reproducing kernel Hilbert space $\cH_K$ is associated to a 
vector valued positive semidefinite kernel for several reasons, but mainly
because this is related to the multiview 
learning, cf.\ \cite{Evgeniou}, \cite{Micchelli}, 
\cite{SindhwaniRosenberg}, \cite{Rosenberg}, \cite{MinhSindhwani},  \cite{Luo}, 
\cite{Kadri}, \cite{MBM}, \cite{Hashimoto21}, \cite{Hashimoto22}, \cite{Hashimoto23a}, \cite{Hashimoto23b}. In 
this article  we consider localised versions of these operator valued reproducing kernel Hilbert 
spaces that offers flexibility for a larger class of learning problems, as explained in the Introduction, 
and does not bring additional obstructions, as proven in Subsection \ref{ss:ovk}.

Further on, the first term in \eqref{e:ielef} controls the distance, estimated by local loss 
(or cost) functions at the labeled input points with respect to the 
labels. More precisely, for each label point $x_j$, $j=1,\ldots,l$ the label 
$y_j\in \cY_{x_j}$ is compared, through the cost function $V_{x_j}$, 
with $f(x_j)\in\cW_{x_j}$ by a 
combination operator $C_{x_j}\colon \cW_{x_j}\ra \cY_{x_j}$, because the Hilbert 
spaces $\cY_{x_j}$ may be different from $\cW_{x_j}$. 

\begin{example}\label{ex:comb} 
{Following \cite{MBM}, for an input point $x\in X$ consider the label 
Hilbert space $\cY$ and let $\cW=\cY^m$, the orthogonal direct sum of $m$ copies of 
$\cY$. With this notation, the kernel $K$ has values in $\cB(\cW)$. A multiview $f(x)$ is 
then an $m$-tuple $(f^1(x),\ldots,f^m(x))^T$, with each $f^i(x)\in\cY$ and let the 
combination operator $C=[C_1,\ldots,C_m]\colon \cW=\cY^m\ra \cY$, that is,
$Cf(x)=C_1f^1(x)+\cdots+C_mf^m(x)\in\cY$.
}
\end{example}

In this article, the loss functions are also localised and one strong reason for this is that, 
depending on different purposes that this semisupervised learning is used for, there is a 
very large pool of choices for loss functions.

\begin{example}\label{ex:lf}  {We list in the following a few loss functions of interest 
in machine learning, see \cite{ZMY} and \cite{Kwon} for a more comprehensive list and 
applications.}\medskip

 {(1)  \emph{Least Squares.} The least squares loss function is
\begin{equation*}
V(y,z)=(y-z)^2,\quad y,z\in \RR.
\end{equation*}
It is convex, nonnegative, and differentiable.}

 {(2)  \emph{Sigmoid.} The sigmoid loss function is
\begin{equation*}
V(y,z)=\frac{1}{1+\exp(z-y)},\quad y,z\in \RR.
\end{equation*}
It is nonnegative, differentiable, and nonconvex.}

 {(3)  \emph{Hinge.} The hinge loss function is
\begin{equation*}
V(y,z)=\max\{0,1-yz\},\quad y,z\in\RR.
\end{equation*}
It is nonnegative, continuous, convex, but not differentiable.}

 {(4)  \emph{Exponential Least Squares.} The exponential least squares function is
\begin{equation*}
V(y,z)=1-\exp(-(y-z)^2),\quad y,z\in \RR.
\end{equation*}
It is nonnegative, upper bounded by $1$, differentiable, and nonconvex.}

 {(5)  \emph{Leaky Hockey Stick.} The leaky hockey stick loss function is
\begin{equation*}
V(y,z)=\begin{cases} -\log(zy),& yz>1,\\ 1-yz,& yz\leq 1.\end{cases}
\end{equation*}
It is upper and lower unbounded, convex, and differentiable.}
\end{example}

The second term in \eqref{e:ielef} is the usual regularisation penalty term, following 
the Tikhonov regularisation method, cf.\ \cite{Tikhonov}. This is used in order to avoid 
large target
functions $f$ and overfitting, that is, optimising functions that match very accurately the 
labeled data but perform badly for other data. Because of this, 
the regularisation parameter $\gamma_A$ is always positive. In the literature, 
sometimes the second term is replaced 
by $\phi(\|f\|_{\cH_K})$, where $\phi\colon \RR_+\ra \RR_+$ is an increasing function so, 
in our case $\phi(t)=\gamma_A t^2$. 

The third term in \eqref{e:ielef} combines vector valued manifold regularisation, 
cf.\ \cite{MinhSindhwani}, with multiview regularisation, cf.\ \cite{Rosenberg} and 
\cite{Sun}. The parameter 
$\gamma_I$ may be taken $1$, without loss of generality, since it can be absorbed in 
$M$.
Following \cite{MBM}, the operator multiview regularisation term 
$\gamma_I \langle \bef,M\bef\rangle_{\bW^{l+u}}$ is decomposed as
\begin{equation} \gamma_I \langle \bef,M\bef\rangle_{\bW^{l+u}}=
\gamma_B \langle \bef,M_B\bef\rangle_{\bW^{l+u}}+\gamma_W \langle 
\bef,M_W\bef\rangle_{\bW^{l+u}},\label{e:gai}
\end{equation}
where $M_B,M_W\in\cB(\bW^{l+u})$ are selfadjoint positive operators and 
$\gamma_B,\gamma_W\geq 0$. As before, the regularising parameters $\gamma_B$ 
and $\gamma_W$ may be taken $1$, without loss of generality, because they can be 
absorbed in $M_B$ and $M_W$, respectively.
The first term in \eqref{e:gai} is the \emph{localised 
between-view regularisation} while the latter term in \eqref{e:gai} is the 
\emph{localised within-view regularisation}. In the next example we show by some 
concrete situations the constructions of the operators $M_B$ and $M_W$ and their 
significance.

\begin{example}\label{ex:bevreg} {This example follows closely the example of
between-view regularisation as in 
\cite{MBM}. With notation as in Example~\ref{ex:comb}, let 
$M_m=mI_m-\mathbf{1}_m\mathbf{1}_m^T$, where $\mathbf{1}_m=(1,1,\ldots,1)^T$. 
More precisely, $M_m$ is the $m\times m$ matrix with all entries equal to $m-1$ throughout its diagonal 
and $-1$ elsewhere. Then, for each $\mathbf{a}=(a_1,\ldots,a_m)^T\in \RR^m$, 
we have
\begin{equation*}
\mathbf{a}^T M_m\mathbf{a}=\sum_{j,k=1,\ j<k}^m (a_j-a_k)^2.
\end{equation*}
Then, for each $\mathbf{y}=(y_1,\ldots,y_m)^T\in\cY^m=\cW$ we have
\begin{equation*}
\mathbf{y}^T (M_m\otimes I_\cY)\mathbf{y}=\sum_{j,k=1,\ j<k} \|y_j-y_k\|^2_\cY.
\end{equation*}
So, letting $M_B=I_{u+l}\otimes (M_m\otimes I_\cY)$, for each $\mathbf{f}=(f(x_1),\ldots,f(x_{u+l})\in \cY^{m(u+l)}=\cW^{u+l}$, with $f(x_i)\in\cY^m=\cW$, we have
\begin{align*} \langle \mathbf{f},M_B\mathbf{f}\rangle_{\cW^{u_l}} & = 
\sum_{i=1}^{u+l} \langle f(x_i),(M_m\otimes I_\cY)f(x_i)\rangle_{\cW} =
\sum_{i=1}^{u+l} \sum_{j,k=1,\ j<k} \|f^j(x_i)-f^k(x_i)\|_\cY^2.
\end{align*}
This term is a control on the consistency  between different components $f^i$'s
which represent the outputs on different views.}
\end{example}

\begin{example}\label{ex:wivreg}
 {This example follows essentially \cite{MBM} for a within-view manifold
regularisation via multiview graph Laplacians in support vector machine learning, 
cf.\ \cite{Sun}. For manifold regularisation, a data adjaceny graph is 
defined in such a way that the entries measure the similarity or closeness of pairs of 
inputs. Given an undirected graph $G=(\cV,E)$, where the vertices are 
$\cV=\{1,\ldots,n\}$ 
and edges are simply pairs $(j,k)$, assume that for each edge $(j,k)\in E$ there is a 
weight $w_{j,k}$, and to each edge $(j,k)\not\in E$ we let $w_{j,k}=0$, 
in such a way that the weight matrix $W=[w_{j,k}]$ is Hermitian 
and nonnegative (positive semidefinite). }

 {For example, when each vertex $j$ is associated to a vector 
$h_j\in \RR^d$, we can use the Gaussian weights
\begin{equation}\label{e:wejek}
w_{j,k}=\exp(-\|h_j-h_k\|^2/2\sigma^2).
\end{equation}
In order to simplify the complexity of calculations,
cf.\ \cite{Sun}, for most of the edges $(j,k)$ we take $w_{j,k}=0$ and only 
for neighbouring $(j,k)$, that is, $\|h_j-h_k\|_2\leq \epsilon$, for some $\epsilon$, 
we define the weights by \eqref{e:wejek}.}

 {Further on, letting $v_{j,j}=\sum_{k=1}^n w_{j,k}$ and 
$v_{j,k}=0$ if $j\neq k$, we make the diagonal matrix $V=[v_{j,k}]$. We work under the 
assumptions that $v_{j,j}>0$ for all $j=1,\ldots,n$. Then the 
\emph{graph Laplacian} matrix is $L:=V-W$, which is positive semidefinite. Sometimes 
it is useful to work with the \emph{normalised graph Laplacian} 
$\widetilde L:=V^{-1/2}LV^{-1/2}$.
}

 {But, because the learning is from multiviews, this should be performed for 
each view and then aggregated in a consistent fashion. From now on we 
use the same notations and 
settings as in Example~\ref{ex:comb} and Example~\ref{ex:bevreg}.
Assume that, to each view $i$, $1\leq i\leq m$, we consider the undirected graph 
$G^i=(\cV^i,E^i)$ where $\cV^i=\{1,\ldots, u+l\}$, let $W^i=[w_{j,k}^i]$ be the 
corresponding weight matrix that is Hermitian and nonnegative, and let $L^i=[l_{j,k}^i]$ 
be the 
corresponding graph Laplacian. Then, for each vector $\mathbf{a}\in\RR^{u+l}$ we have
\begin{equation*}
\mathbf{a}^T L^i \mathbf{a}=\sum_{j,k=1,\ j<k}^{l+u} w_{j,k}^i (a_j-a_k)^2.
\end{equation*}
Now we aggregate the graph Laplacians into the multiview graph Laplacian
as a block matrix $L=[L_{j,k}]$, where for each $j,k=1,\ldots,u+l$ we define
\begin{equation*}
L_{j,k}=\diag(l^1_{j,k},\ldots,l^m_{j,k}).
\end{equation*}
This implies that for each vector $\mathbf{a}=(a_1,\ldots,a_{u+l})$, 
with $a_j=(a_j^1,\ldots,a_j^m)\in\RR^m$ for each $j=1,\ldots,u+l$, we have
\begin{equation*}
\mathbf{a}^TL\mathbf{a}=\sum_{i=1}^m \sum_{j,k=1,\ j<k}^{l+u} w_{j,k}^i (a_j^i-a_k^i)^2.
\end{equation*}
Finally, letting $M_W:=L\otimes I_\cY$, we have
\begin{equation*}
\langle \mathbf{f},M_W\mathbf{f}\rangle_{\cW^{u+l}}=\sum_{i=1}^m \sum_{j,k=1,\ j<k}^{l+u}
w_{j,k}^i \|f^i(x_j)-f^i(x_k)\|^2_\cY,
\end{equation*}
for all $\mathbf{f}=\{f^i(x_j)\mid i=1,\ldots,m,\ j=1,\ldots,u+l\}\in\cW^{u+l}=\cY^{m(u+l)}$.
Each term in the leftmost sum is a manifold regularisation for the view $i$ and 
hence the double sum is the aggregated manifold regularisation for all views. In this 
fashion, consistency is enforced for each view.
}
\end{example}

\subsection{A Representer Theorem.}\label{ss:agrt}
We continue to use the notation as in the previous subsection.
Generally speaking, a representer theorem has the goal to prove that 
the optimal solution to the problem \eqref{e:fezga} should belong to the space
\begin{equation}\label{e:hekab}
    \mathcal{H}_{K,\bx} = \Bigl\{ \sum^{l+u}_{i=1} K_{x_i}w_i \mid w_i \in \mathcal{W}_{x_i} \Bigr\}.
\end{equation}
Let $\ol\cH_{K,\bx}$ denote its closure in $\cH_{K}$ and let $P_{\ol\cH_{K,\bx}}$ denote the orthogonal projection
of $\cH_K$ onto $\ol\cH_{K,\bx}$. 
Let the sampling operator $S_{\mathbf{x}}: \mathcal{H}_K \to \bW^{l+u}$ be defined by
\begin{equation}\label{e:so}
    S_{\mathbf{x}}f = (K_{x_i}^* f)_{i=1}^{l+u}=(f(x_i))_{i=1}^{l+u} = \mathbf{f},\quad f\in \cH_K,
\end{equation}    
where $\bx=(x_1,\ldots,x_{l+u})$ and we have taken into account that $f(x) = K^*_{x}f$ for all $f\in \cH_K$ and all $x\in X$. 
 Let also $E_{\bC,{\mathbf{x}}}:\mathcal{H}_K \to \bY^l$, where 
 \begin{equation}\label{e:beyel}\bY^l:=\bigoplus_{j=1}^l \cY_{x_j},\end{equation} be defined by
\begin{equation}\label{e:ecf}
E_{\bC,\bx}f = \left( C_{x_1}K_{x_1}^*f, \dots,C_{x_l}K_{x_l}^*f\right)=
\left( C_{x_1}f(x_1), \dots,C_{x_l}f(x_l)\right),\quad f\in \cH_K.
\end{equation}

The main technical fact used in this section is a lemma whose proof is inspired 
by the proof of Theorem~2 in \cite{MBM}. 

\begin{lemma}\label{l:projection} With notation and assumptions an in 
Subsection~\ref{ss:grmlp} and as before, 
for any $f\in\cH_K$ the following inequality holds.
\begin{equation*}
\cI(P_{\ol\cH_{K,\bx}}f)\leq \cI(f).
\end{equation*}
\end{lemma}

\begin{proof}  We have the decomposition
\begin{equation}\label{e:decor}
\cH_K=\ol\cH_{K,\bx}\oplus \cH_{K,\bx}^\perp.
\end{equation}
Let $f\in \mathcal{H}_{K,\bx}^\bot$ be fixed. Then, for any $\mathbf{b}\in \bY^l$,
since $C_{x_i}^*b_i \in \mathcal{W}_{x_i}$, for all $i=1,\ldots,l+u$, and hence
\begin{equation*}
\sum_{i=1}^l K_{x_i}C_{x_i}^* b_i\in \cH_{K,\bx},
\end{equation*} 
we have
\begin{equation*}
    \innerprod{E_{\bC,\mathbf{x}}f}{\mathbf{b}}_{\bY^l} = \innerprod{f}{E_{C,\mathbf{x}}^*\mathbf{b}}_{\mathcal{H}_K} = \sum^{l}_{i=1} \innerprod{f}{K_{x_i}C_{x_i}^*b_i}_{\mathcal{H}_K} =
    \langle f,\sum_{i=1}^l K_{x_i}C_{x_i}^* b_i\rangle_{\cH_K}= 0.
\end{equation*}
 Consequently, 
\begin{equation}\label{e:bot}
E_{\bC,\mathbf{x}}f = (C_{x_1}K_{x_1}^*f, \dots, C_{x_l}K_{x_l}^*f) = 0.
\end{equation}

Similarly, by the reproducing property, letting $\bw=(w_1,\ldots,w_{l+u})$ be an arbitrary vector in $\bW^{l+u}$, we have
\begin{equation*}
    \innerprod{S_\mathbf{x}f}{\mathbf{w}}_{\bW^{l+u}} = \sum^{l+u}_{i=1}\innerprod{f(x_i)}{\mathbf{w}}_{\bW^{l+u}} =
  \sum^{l+u}_{i=1} \langle f, K_{x_i}w_i\rangle_{\cH_K}  
    = \innerprod{f}{\sum^{l+u}_{i=1}K_{x_i}w_i}_{\mathcal{H}_K} = 0,
\end{equation*}
hence 
\begin{equation}\label{e:botex}
\mathbf{f}= S_\mathbf{x}f = (f(x_1),\dots,f(x_{l+u})) = 0.
\end{equation}

For arbitrary $f\in\mathcal{H}_K$, in view of the decomposition \eqref{e:decor}, we 
have the unique decomposition
$f = f_0+f_1$ with $f_0\in \ol\cH_{K,\mathbf{x}}$ and $f_1\in \mathcal{H}_{K,\mathbf{x}}^\bot$, that is,
$f_0=P_{\ol\cH_{K,\bx}}f$.
Then, 
\begin{equation*}\norm{f_0+f_1}^2_{\mathcal{H}_K} = \norm{f_0}_{\mathcal{H}_K}^2 
+ \norm{f_1}_{\mathcal{H}_K}^2,\end{equation*}  
and consequently,
\begin{align}
    \mathcal{I}(f) &= \mathcal{I}(f_0+f_1) = \dfrac{1}{l} \sum_{i = 1}^l V_{x_i}(y_i, C_{x_i}K^*_{x_i}f_0 + C_{x_i}K^*_{x_i}f_1)+ \gamma_A \norm{f_0}^2_{\mathcal{H_K}} + \gamma_A \norm{f_1}^2_{\mathcal{H_K}}\nonumber \\
 &\ \ \ + \gamma_I \innerprod{S_\mathbf{x}f_0}{MS_\mathbf{x}f_0}_{\bW^{l+u}} + \gamma_I \innerprod{S_\mathbf{x}f_0}{MS_\mathbf{x}f_1}_{\bW^{l+u}} \nonumber\\
 &\ \ \ + \gamma_I \innerprod{S_\mathbf{x}f_1}{MS_\mathbf{x}f_0}_{\bW^{l+u}} +\gamma_I \innerprod{S_\mathbf{x}f_1}{MS_\mathbf{x}f_1}_{\bW^{l+u}}.\label{e:iafaz}
\end{align}
By \eqref{e:bot} we then see that
\begin{equation*}\label{e:bound}
    V_{x_i}(y_i, C_{x_i}K^*_{x_i}f_0 + C_{x_i}K^*_{x_i}f_1) = V_{x_i}(y_i, C_{x_i}K^*_{x_i}f_0),
\end{equation*}
and by \eqref{e:botex}  we see that
\begin{equation*}
    \innerprod{S_\mathbf{x}f_0}{MS_\mathbf{x}f_1}_{\bW^{l+u}} =  \innerprod{S_\mathbf{x}f_0}{0}_{\bW^{l+u}} = 0.
\end{equation*}
So,  
\begin{equation}\label{e:manbound}
    \innerprod{S_\mathbf{x}f_0}{MS_\mathbf{x}f_1}_{\bW^{l+u}} = \innerprod{S_\mathbf{x}f_1}{MS_\mathbf{x}f_0}_{\bW^{l+u}} = \innerprod{S_\mathbf{x}f_1}{MS_\mathbf{x}f_1}_{\bW^{l+u}} = 0
\end{equation}
and hence, by \eqref{e:iafaz}, we have
\begin{align}
    \mathcal{I}(f) &= \mathcal{I}(f_0+f_1) = \dfrac{1}{l} \sum_{i = 1}^l V_{x_i}(y_i, C_{x_i}K^*_{x_i}f_0 + C_{x_i}K^*_{x_i}f_1)+ \gamma_A \norm{f_0}^2_{\mathcal{H_K}} + \gamma_A \norm{f_1}^2_{\mathcal{H_K}} \nonumber \\
    &\ \ \ + \gamma_I \innerprod{S_\mathbf{x}f_0}{MS_\mathbf{x}f_0}_{\bW^{l+u}} + \gamma_I \innerprod{S_\mathbf{x}f_0}{MS_\mathbf{x}f_1}_{\bW^{l+u}} + \nonumber \\
    &\ \ \ + \gamma_I \innerprod{S_\mathbf{x}f_1}{MS_\mathbf{x}f_0}_{\bW^{l+u}} +\gamma_I \innerprod{S_\mathbf{x}f_1}{MS_\mathbf{x}f_1}_{\bW^{l+u}} \nonumber \\
\intertext{and then, by \eqref{e:bot} and \eqref{e:manbound}
we get that}
\mathcal{I}(f) & =    \dfrac{1}{l} \sum_{i = 1}^l V_{x_i}(y_i, C_{x_i}K^*_{x_i}f_0) + \gamma_A\norm{f_0}^2_{\mathcal{H_K}} + \gamma_A \norm{f_1}^2_{\mathcal{H_K}}+ \gamma_I\innerprod{S_\mathbf{x}f_0}{MS_\mathbf{x}f_0}_{\bW^{l+u}}
\nonumber \\
&\geq 
    \dfrac{1}{l} \sum_{i = 1}^l V_{x_i}(y_i, C_{x_i}K^*_{x_i}f_0) + \gamma_A\norm{f_0}^2_{\mathcal{H_K}} + \gamma_I\innerprod{S_\mathbf{x}f_0}{MS_\mathbf{x}f_0}_{\bW^{l+u}}=\cI(f_0),\label{e:desiv}
\end{align}
and the proof is finished.
\end{proof}

In order to get a conclusion in the spirit of the representer theorem, extra assumptions are needed.

\begin{proposition}\label{p:minimizer} Assume that the subspace $\cH_{K,\bx}$, see \eqref{e:hekab}, is closed.
This happens, for example, if 
all Hilbert spaces $\cW_{x_1},\ldots,\cW_{x_{l+u}}$ have finite dimensions. If the minimisation problem 
\eqref{e:fezga} has a solution $f_{\bz,\gamma}$ then there exist 
$a_1,\ldots,a_{l+u}$, with $a_j\in \cW_{x_j}$ for all $j=1,\ldots,l+u$, such that
\begin{equation*}
f_{\bz,\gamma}=\sum_{j=1}^{l+u} K_{x_j} a_j.
\end{equation*}
\end{proposition}

\begin{proof} Since $\cH_{K,\bx}$ is closed, we have the decomposition
\begin{equation}\label{e:decora}
\cH_K=\cH_{K,\bx}\oplus \cH_{K,\bx}^\perp.
\end{equation}
If $f$ is a solution to the minimisation problem \eqref{e:fezga}, in view of Lemma~\ref{l:projection}, it follows 
that $f\in \cH_{K,\bz}$, and the conclusion follows.
\end{proof}

The main theorem of this section is a representer theorem under certain general and natural 
assumptions.

\begin{theorem}\label{t:fdcont}
Assume that the loss functions $V_{x_j}(y_j,\cdot)$ 
are bounded from below and continuous, for all $j=1,\ldots,l$, 
and that all Hilbert spaces $\cW_{x_1},\ldots,\cW_{x_{l+u}}$ have finite dimensions. 
Then the minimisation problem 
\eqref{e:fezga} has a solution $f_{\bz,\gamma}$ and, for any such a solution, there exist 
$a_1,\ldots,a_{l+u}$, with $a_j\in \cW_{x_j}$ for all $j=1,\ldots,l+u$, such that
\begin{equation*}
f_{\bz,\gamma}=\sum_{j=1}^{l+u} K_{x_j} a_j.
\end{equation*}
\end{theorem}

\begin{proof} We first observe that, since all loss functions $V_{x_j}(y_j,\cdot)$, $j=1,\ldots,l+u$, are lower 
bounded and $M$ is positive semidefinite, from \eqref{e:ielef}
it follows that $\cI(f)$ is lower bounded and hence
its infimum exists as a real number. Since all the Hilbert spaces $\cW_{x_1},\ldots,\cW_{x_{l+u}}$ have 
finite dimensions it follows that the subspace $\cH_{K,\bx}$ has finite dimension and hence it is closed.
Then, from Lemma~\ref{l:projection} it follows that
\begin{equation*}
-\infty<\inf_{f\in \cH_K} \cI(f)=\inf_{f\in\cH_{K,\bx}} \cI(f).
\end{equation*}
So, it remains only to show that the latter infimum is attained.

Indeed, since $\gamma_A>0$ and the loss functions $V_{x_j}(y_j,\cdot)$ 
are bounded from below for all $j=1,\ldots,l$,
 it follows that 
\begin{equation}\label{e:lim}
\lim_{\|f\|_{\cH_K}\ra\infty} \cI(f)=+\infty.
\end{equation}
Since all loss functions $V_{x_j}(y_j,\cdot)$ are continuous, for $j=1,\ldots,l$, and the evaluation 
functionals on $\cH_K$ are continuous as well, see Theorem~\ref{t:evop}, 
it follows that $\cI$ is continuous on $\cH_K$. 
Let $f_0\in\cH_{K,\bx}$ be arbitrary but fixed. From \eqref{e:lim}, for $\epsilon>0$ there exists $\delta>0$ 
such that 
\begin{equation}\label{e:cig}
\cI(f)\geq \cI(f_0)+\epsilon\mbox{ for all }f\in\cH_{K,\bx}\mbox{ with }\|f-f_0\|_{\cH_K}>\delta.\end{equation} 
We consider now the
continuous function $\cI$ restricted to the closed ball in $\cH_{K,\bx}$
\begin{equation*}\ol B_\delta^{\cH_{K,\bx}}(f_0)=\{f\in\cH_{K,\bx}\mid \|f-f_0\|_{\cH_K}\leq \delta\},\end{equation*} 
which is compact, since the vector space $\cH_{K,\bx}$ is finite dimensional. 
This implies that the infimum of $\cI$ on $\ol B_\delta^{\cH_{K,\bx}}(f_0)$ is attained. In view of \eqref{e:cig} it follows
that
\begin{equation*}
\inf_{f\in \cH_{K,\bx}}\cI(f)=\inf\{\cI(f)\mid f\in \ol B_\delta^{\cH_{K,\bx}}(f_0)\},
\end{equation*}
and the proof is finished.
\end{proof}

In view of Proposition~\ref{p:minimizer}, if the loss functions $V_{x_i}(y_i,\cdot)$ are
convex for all $i=1,2,\ldots,l$, then the assumption on finite dimensionality of the spaces $\cW_{x_i}$ 
for $i=1,\ldots,l+u$, can be slightly weaker. We first record a result that is known but for which 
we include a proof, for the reader's convenience. To this end, we recall some basic definitions.
  If $\cV$ is a vector space, a subset $A\subseteq \cV$ is \emph{convex} if for any $x,y\in A$ and $\lambda\in (0,1)$ we have $(1-\lambda)x+\lambda y\in A$.
  A function $f:\dom(f) \to \mathbb{R}$ is \emph{convex} if $\dom(f)$ is a convex set in $\cV$ and
  \begin{equation*}
      f(\lambda x + (1-\lambda)y) \leq \lambda f(x) + (1-\lambda)f(y)
  \end{equation*}
  for each $x,y \in \dom(f)$ and $\lambda \in (0,1)$. 
  In addition, $f$ is \emph{strictly convex} if
  \begin{equation*}
      f(\lambda x + (1-\lambda)y) < \lambda f(x) + (1-\lambda)f(y)
  \end{equation*}
  for each $x,y \in \dom(f)$ such that $x\neq y$ and $\lambda \in (0,1)$.

\begin{lemma}\label{l:semconv}
\emph{(a)} Given a vector space $\cV$ that is endowed with a seminorm $\norm{\cdot}$, 
the square of the seminorm $\norm{\cdot} ^2\colon \cV \to \mathbb{R}$ is a convex function on $\cV$. 

\emph{(b)} If, in addition, $\norm{\cdot} $ is a norm associated to an inner product on the real vector space $\cV$, 
then $\norm{\cdot} ^2$ becomes strictly convex.
\end{lemma}

\begin{proof} (a)
  Let $f_0, f_1 \in \cV, \alpha \in (0,1)$. Then by triangle inequality we have
\begin{align*}
    \norm{\alpha f_0 + (1-\alpha)f_1}  &\leq 
    \norm{\alpha f_0}  +  \norm{(1-\alpha)f_1}  = \alpha\norm{ f_0}  + (1-\alpha)\norm{f_1} ,
\end{align*}
hence, by squaring both sides we get
\begin{equation*}
    \norm{\alpha f_0 + (1-\alpha)f_1} ^2 \leq \alpha^2\norm{ f_0}^2 + (1-\alpha)^2\norm{f_1}^2 + 2\alpha(1-\alpha)\norm{ f_0} \norm{f_1} .
\end{equation*}
Further on, if we add and subtract $-\alpha\norm{f_0}^2 - (1-\alpha)\norm{f_1}^2$ from the 
right hand side, we get
\begin{align}
\norm{\alpha f_0 + (1-\alpha)f_1}^2 &\leq \alpha^2\norm{ f_0}^2 + (1-\alpha)^2\norm{f_1}^2 
 + 2\alpha(1-\alpha)\norm{ f_0} \norm{f_1}  \nonumber \\
& \ \ \ \   -\alpha\norm{f_0}^2 - (1-\alpha)\norm{f_1} ^2 + \alpha\norm{f_0}^2 
+ (1-\alpha)\norm{f_1}^2\nonumber \\
& = (\alpha^2-\alpha)\norm{ f_0}^2 + ((1-\alpha)^2-(1-\alpha))\norm{f_1}^2  \nonumber \\
& \ \ \ \ +2\alpha(1-\alpha)\norm{ f_0} \norm{f_1}  + \alpha\norm{f_0}^2 + (1-\alpha)\norm{f_1}^2 
\nonumber \\
& = -\alpha(1-\alpha)(\norm{f_0} -\norm{f_1} )^2 + \alpha\norm{f_0}^2 + (1-\alpha)\norm{f_1}^2
\nonumber \\
& \leq  \alpha\norm{f_0}^2 + (1-\alpha)\norm{f_1}^2.
\label{e:scn}
\end{align}
This shows that $\norm{\cdot}^2 $ is convex.

(b) We assume now that $\norm{\cdot} $ is a norm associated to an inner product $\langle\cdot,\cdot\rangle$ 
on a real vector space $\cV$, that $f_0\neq f_1$, $\alpha\in (0,1)$, and that
\begin{equation}\label{e:naf}
    \norm{\alpha f_0 + (1-\alpha)f_1}^2  =  \alpha \norm{f_0}^2 + (1-\alpha)\norm{f_1}^2 ,
\end{equation} hence, by the last step in \eqref{e:scn}, it follows that 
$\norm{f_0}  = \norm{f_1} $. Then, by \eqref{e:naf} and
since $\norm{f_0}  = \norm{f_1} $, we get
\begin{align*}
  \norm{f_0}^2 & = \norm{\alpha f_0 + (1-\alpha)f_1}^2  = \langle \alpha f_0+(1-\alpha)f_1,\alpha f_0+(1-\alpha)f_1\rangle \\
  & = \alpha^2\norm{f_0}^2+2\alpha(1-\alpha)\langle f_0,f_1\rangle + (1-\alpha)^2 \norm{f_1}^2.  
\end{align*}
Taking into account that $\norm{f_0}=\norm{f_1}$ and that $\alpha(1-\alpha)\neq 0$, 
from here it follows that 
\begin{equation}\label{e:laf}\langle f_0,f_1\rangle=\norm{f_0}^2=\norm{f_1}^2,\end{equation}
hence we have equality in the Schwarz inequality and, consequently, $f_0=t f_1$ for some $t\in \RR$. Since 
$\norm{f_0}=\norm{f_1}$ it follows that $t=\pm 1$. But $t=1$ is not possible since $f_0\neq f_1$, while $f_0=-f_1$ is 
not possible because, by \eqref{e:laf}, this would imply $f_0=0=f_1$.
\end{proof}

\begin{theorem}\label{t:fdconv} Assume that all the underlying vector spaces are real, 
that the subspace $\cH_{K,\bx}$ is closed,
and that the loss functions $V_{x_i}(y_i,\cdot)$ are
convex for all $i=1,2,\ldots,l$. Then, the minimisation problem \eqref{e:fezga} has
a unique solution $f_{\bz,\gamma}$ and there exist 
$a_1,\ldots,a_{l+u}$, with $a_j\in \cW_{x_j}$ for all $j=1,\ldots,l+u$, such that
\begin{equation*}
f_{\bz,\gamma}=\sum_{j=1}^{l+u} K_{x_j} a_j.
\end{equation*}
\end{theorem}

\begin{proof} Consider the function $\bV^l\colon \bY^l\times \bY^l\ra \RR$ defined by
\begin{equation}\label{e:bvl}
\bV^l(\by,\by^\prime):=\sum_{j=1}^l V_{x_j}(y_j,y_j^\prime),\quad \by=(y_1,\ldots,y_l),\ \by^\prime=(y_1^\prime,\ldots,
y_l^\prime),
\end{equation}
and observe that, for each fixed $\by\in \bY$, the function $\bV^l(\by,\cdot)$ is convex on $\bY$, since all maps 
$V_{x_j}$ are convex in the second argument, $j=1,\ldots,l$. Consequently, in the definition of $\cI$ at
\eqref{e:ielef}, the first term is a convex function. Since the second term is a norm, it is a strictly convex function, 
while the third term is a seminorm, hence a convex function as well, by Lemma~\ref{l:semconv}. 
Thus, $\cI$ is a strictly convex function and hence
the minimisation problem \eqref{e:fezga} has a unique solution. Then the conclusion follows from 
Proposition~\ref{p:minimizer}.
\end{proof}

\begin{remark}\label{r:inf} 
 {Theorem~\ref{t:fdconv} contains Theorem 2 in \cite{MBM} in the case when 
the subspace $\cH_{K,\bx}$, see \eqref{e:hekab}, is closed. This happens, for example, if
the Hilbert space $\cW$ in that theorem is finite dimensional. In 
\cite{MBM} the authors claim that the result is true even in the case when $\cW$ is an infinite dimensional space, 
which is not substantiated by the proof they provide. More precisely, the gap in that proof is that the subspace
$\cH_{K,\bx}$ might not be finite dimensional and hence it might not be closed, which implies that, we have the
decomposition \eqref{e:decor} and not the decomposition \eqref{e:decora}.
Consequently, the only conclusion that can be drawn
is that the minimiser $f_{\bz,\gamma}$ belongs to the closure of $\cH_{K,\bx}$, and hence can only be 
approximated in the norm of $\cH_K$ by sums of 
type $\sum_{j=1}^{l+u} K_{x_j} a_j$, but it may never equal such a sum.}
\end{remark}

\section{Differentiable Loss Functions} 
\subsection{Preliminary Results on Differentiable Optimisation.}
Throughout this section, we assume that all vector spaces are real. The definitions and proofs of facts 
recalled in this 
subsection are from \cite{Peypouquet}.  
 If $\cX$ is a normed space, the \emph{directional derivative} of a function 
  $f\colon \dom(f)(\subseteq \cX)\to\mathbb{R}$ at an interior point $x\in \dom(f)$ in the direction $h\in \cX$ is given by
  \begin{equation*}
      f'(x:h) = \lim_{t \to 0} \dfrac{f(x+th)-f(x)}{t},
  \end{equation*}
  provided that the limit exists.
 A function $f\colon \dom(f)(\subseteq \cX) \to \mathbb{R}$ is \emph{Gâteaux differentiable} at an interior 
point $x\in \dom(f)$ if $f$ has directional 
derivatives for all directions at $x$ and $\phi_x(h):=f'(x:h)$ is linear and continuous in $h$. In this case, we 
denote the \emph{Gâteaux derivative} $\nabla_x f\in \cB(\cX,\RR)=X^*$ by the \emph{gradient} notation
    \begin{equation*}
        (\nabla_x f)h := \phi_x(h),\quad h\in \cX.
    \end{equation*}
  
    In general, if $\cX$ and $\cY$ are Banach real spaces and $U\subseteq \cX$ is open and 
    $F:\cX\to \cY$, then $F$ \emph{has directional derivative for all directions at point $x\in U$} if
    \begin{equation*}
        \lim_{\tau \to 0}\dfrac{F(x+\tau h)-F(x)}{\tau}
    \end{equation*}
    exists for any $h\in \cX$. In this case we define the map $\nabla_x F\colon \cX \to \cY$ as 
    \begin{equation*}
       ( \nabla_x F)h:= \lim_{\tau \to 0}\dfrac{F(x+\tau h)-F(x)}{\tau},\quad h\in \cX.
    \end{equation*}
 
 In the following we recall some basic facts.
 
 \begin{theorem}[Chain Rule for Directional Derivative]\label{t:chain}
 Assume that $\cX,\cY,\cZ$ are Banach spaces, $F\colon \cX\to \cY$, $G\colon \cY\to \cZ$ and there exists $U\subseteq \cX$ and 
 $V \subseteq \cY$ open such that $F(U)\subseteq \cV$, 
 $G$ has directional derivatives for all directions at $y \in \cV$, and $F$ has 
 directional derivatives for all directions at $x\in U$. 
 If $\nabla_x F$ and $\nabla_y G$ are continuous in $x\in U$ and $y\in V$, respectively, then, for any 
 $h\in \cX$,
  we have
  \begin{equation*}
      \nabla_x(G\circ F)(h) = \nabla_{F(x)}G(\nabla_x F(h)).
  \end{equation*}
  \end{theorem}

Let $\cX$ be a Banach real space. For any fixed $x^*\in \cX^*:=\mathcal{B}(\cX;\mathbb{R})$ and any 
$y\in \cX$ we 
denote 
\begin{equation}\label{e:dbracket}\innerprod{x^*}{y} := x^*(y) \in \mathbb{R}.\end{equation}

\begin{remark}
 {If $\cH$ is a real Hilbert space, by Riesz-Fr\'echet Representation Theorem we have $\innerprod{y}{x^*} 
= \innerprod{y}{f_{x^*}}_{\cH}$ for a unique $f_{x^*}\in \cH$. For simplicity we denote $f_{x^*}$ as $ x^*$.}
\end{remark}

Let $f\colon \dom(f)(\subseteq \cX)\to \mathbb{R}$ be convex. A point $x^*\in \cX^*$ is a  \emph{subgradient} 
of $f$ at $x$ if
       \begin{equation*}
           f(y) \leq f(x) + \innerprod{x^*}{y-x}
       \end{equation*}
holds for all $y$ in a neighbourhood of $x$. The set of all subgradients of $f$ at $x$ is the \emph{subdifferential} of 
$f$ at $x$ and is denoted by $\partial f(x)$. If $\partial f(x) \neq \emptyset$, we say $f$ is \emph{subdifferentiable} at 
$x$. The \emph{domain of subdifferential} is denoted as $\dom(\partial f) = \{x\in X \mid \partial f(x) \neq \emptyset\}$. 
By definition $\dom(\partial f) \subseteq \dom(f)$.
    
    \begin{theorem}[Fermat's Rule]\label{t:fermat}
     Let $f\colon \cX\to \mathbb{R}$ be convex. Then $\hat{x}$ is a global minimiser of $f$ if and only if $0\in\partial f(\hat{x})$.
    \end{theorem}

   Let $A$ be an open subset of $\cX$ and $T:A\to \cY$. Then given a point $x \in A$, $T$ is \emph{Fréchet differentiable} at $x$ if there exists a bounded linear operator $L_x\colon \cX\to \cY$ such that
   \begin{equation*}
       \lim_{h\to 0}\dfrac{\norm{T(x+h)-T(x)-L_xh}_{Y}}{\norm{h}_{X}} = 0.
   \end{equation*}
   In this case, due to the uniqueness of $L_x$ we define the linear bounded operator 
   $\deriv_x T:\cX\to \cY$ as
       $\deriv_x T := L_x$
   and call it the \emph{Fréchet derivative} of $T$ at $x$. 

\begin{theorem}\label{t:derip}
Let $(\mathcal{H}, \innerprod{.}{.}_\mathcal{H})$ be a Hilbert real space. Then, given $\psi \in \cX$ and $F,G\colon \cX\to \mathcal{H}$ maps that are Fréchet differentiable at $\psi$, we have
\begin{equation*}
    \deriv_\psi\innerprod{F\cdot}{G\cdot}_\mathcal{H}(h) = \innerprod{(\deriv_\psi F)(h)}{G(\psi)}_\mathcal{H} + \innerprod{F(\psi)}{(\deriv_\psi G)(h)}_\mathcal{H}
\end{equation*} for any $h\in \mathcal{H}$.
\end{theorem}

\begin{remark}\label{Derivative of Composition of Linear Operators and inner product}
   {In Theorem \ref{t:derip} if we further assume that $F,G$ are bounded linear operators, then we get
  \begin{equation*}
      \deriv_\psi \innerprod{F\cdot}{G\cdot}_{\mathcal{H}}(h) = \innerprod{Fh}{G\psi}_{\mathcal{H}} + \innerprod{F\psi}{Gh}_{\mathcal{H}}
  \end{equation*} for any $h \in \mathcal{H}$.}
\end{remark}

\begin{theorem}\label{t:gatfre}
Let $\cX,\cY$ be Banach spaces. Consider a nonempty open set $U\subset \cX$ and a map 
$F\colon U\to \cY$. If $F$ is 
Fr\'echet differentiable at $x\in U$ then it is Gâteaux differentiable at $x$ and $\nabla_x F = \deriv_x F$.
\end{theorem}

\subsection{The Representer Theorem for Locally Differentiable Loss Functions.} \label{ss:rtldlf}
The notation in this subsection is the same as in Section~\ref{s:rmp},
only that all vector spaces are over the real field $\RR$. 
In this subsection, we show that, if we add the assumption that for any $i=1,\ldots,l$, 
$V_{x_i}(y_i,\cdot): \cY_x \to \mathbb{R}$ is Gâteaux differentiable, 
then we can allow $\cW_{x_i}$, for all 
$i=1,\ldots,l$, to be infinite dimensional and get the same conclusion 
as in Proposition~\ref{p:minimizer}. The proof of this theorem is inspired to a certain extent
by the proof of Theorem~3 in \cite{MBM}, that was proven for the special case of the least squares loss function.

As in the proof of Lemma~\ref{l:projection}, let the sampling operator 
$S_{\mathbf{x}}: \mathcal{H}_K \to \bW^{l+u}$ be defined as in \eqref{e:so}   
where $\bx=(x_1,\ldots,x_{l+u})$. 
Let also $E_{\bC,{\mathbf{x}}}:\mathcal{H}_K \to \bY^l$ be defined as in \eqref{e:ecf} and the Hilbert space
$\bY^l$ be defined as in \eqref{e:beyel}. Define the function $\bV^l\colon \bY^l\times \bY^l\ra \RR$ as in \eqref{e:bvl}
and then denote $\bV_{\by}^l\colon \bY^l\ra \RR$ as
\begin{equation}\label{e:bevey}
    \bV_{\by}^l(\by') := \bV^l(\by,\by^\prime),\quad \by^\prime=(y_1^\prime,\ldots,
y_l^\prime)\in \bY^l,
\end{equation}
where $\by=(y_1,y_2,\ldots,y_l)\in \bY^l$.

In the next theorem, which is the main result of this section, please note that
all Hilbert spaces $\cW_{x_1},\ldots,\cW_{x_{l+u}}$ are allowed to be infinite 
dimensional, without any restriction except that they are real.

\begin{theorem}\label{t:gatmin}
Assume that for any $i=1,\ldots,l,$ the loss function
$V_{x_i}(y_i,\cdot)$ is Gâteaux differentiable. If the minimisation problem \eqref{e:fezga} 
has a solution $f_{\bz,\gamma} \in \cH_K$ then there exist $a_1,\ldots,a_{l+u}$, 
with $a_j\in \cW_{x_j}$ for all $j=1,\ldots,l+u$, such that
\begin{equation}\label{e:fezagam}
f_{\bz,\gamma}=\sum_{j=1}^{l+u} K_{x_j} a_j,
\end{equation}
where the vectors $a_i\in\cW_{x_i}$, $i=1,\ldots,l+u$, satisfy the following system
\begin{align}\label{e:minval1}
2\gamma_A l a_i  +2\gamma_I l \sum_{j,k=1}^{l+u}M_{i,j}K(x_i,x_j)a_j & = 
- C_{x_i}^*(\nabla_{E_{\bC,\bx}{f_{\bz,\gamma}}}\bV_{\by}^l)_i,& \mbox{ if }i=1,\dots,l,\\
\gamma_A  a_i  +\gamma_I \sum_{j,k=1}^{l+u}M_{i,j}K(x_i,x_j)a_j & = 0,& \mbox{ if }i=l+1,\dots,l+u.\label{e:minval2}
\end{align}
Here, for arbitrary $y'\in \bY^l$, we abuse some notation and denote $\nabla_{y'}\bV_{\by}^l$ both as the original 
functional (Gâteaux derivative) and the vector that represents this functional following the notation as in 
\eqref{e:dbracket}.
\end{theorem}

\begin{proof} We can rewrite the map $\cI$ to be minimised, see \eqref{e:ielef}, as
\begin{equation}
    \cI(f) = \frac{1}{l}\bV_{\by}^l(E_{\bC,\bx}f) + \gamma_A \|f\|_{\cH_K}^2+ \gamma_I \langle \bef,M\bef\rangle_{\bW^{l+u}},\quad f\in\cH_K.\label{e:opcf}
\end{equation}
Assuming that the minimisation problem 
\eqref{e:fezga} has a solution $f_{\bz,\gamma}$, which is an interior point in the domain of $\cI$ and that
$\bV_{\by}^l$ is Gâteaux differentiable, by Theorem $\ref{t:chain}$, the Gâteaux derivative 
evaluated at $f\in \cH_K$ is
\begin{equation*}
\nabla_f(\bV_{\by}^l\circ E_{\bC,\bx})(h) = \nabla_{E_{\bC,\bx}f}\bV_{\by}^l(\nabla_f E_{\bC,\bx}(h)),\quad 
h\in \cH_K.
\end{equation*}
Then, by using the identification as in \eqref{e:dbracket} and Theorem \ref{t:gatfre}, for all $h\in \cH_K$ we get
\begin{align}
     \nabla_{E_{\bC,\bx}f}\bV_{\by}^l(\nabla_f E_{\bC,\bx}(h)) 
      & = \innerprod{\nabla_{E_{\bC,\bx}f}\bV_{\by}^l}{\deriv_f E_{C,\mathbf{x}}h} \nonumber
     =\innerprod{\nabla_{E_{\bC,\bx}f}\bV_{\by}^l}{E_{C,\mathbf{x}}h}_{\bY^l} \nonumber \\
     &= \innerprod{E_{C,\mathbf{x}}^*\nabla_{E_{\bC,\bx}f}\bV_{\by}^l}{h}_\mathcal{H_K}.\label{e:dercf}
\end{align}
Since 
$\norm{f}^2_{\cH_K} = \innerprod{f}{f}_{\cH_K}$, by Theorem~\ref{t:derip},
the map $\cH_k\mapsto \|f\|_{\cH_K}^2$ is Fr\'echet differentiable and hence, by Theorem~\ref{t:gatfre}, 
it is G\^ateaux differentiable.
Thus, by Theorem \ref{t:derip} and since $\cH_K$ is real we have
\begin{equation}
    \gamma_A \nabla_f\norm{\cdot}^2_{\cH_K}(h) = 2\gamma_A\innerprod{f}{h}_{\cH_K},\quad h\in \cH_K.
\label{e:dernorm}
\end{equation}
Again by using Theorem \ref{t:derip} we get
\begin{align}
    \gamma_I \nabla_f\innerprod{S_{\bx} \cdot}{MS_{\bx} \cdot}_{\bW^{l+u}}(h) &=\gamma_I \innerprod{S_{\bx}f}{MS_{\bx}h}_{\bW^{l+u}} + \gamma_I \innerprod{S_{\bx}h}{MS_{\bx}f}_{\bW^{l+u}} \nonumber
    \\&= 2\gamma_I\innerprod{S_{\bx}^*MS_{\bx}f}{h}_{\cH_K},\quad h\in \cH_K.\label{e:derman}
\end{align}

Summing up \eqref{e:dercf}, \eqref{e:dernorm} and \eqref{e:derman} we get
\begin{align}
    \nabla_f\cI(h) &= l^{-1} \innerprod{E_{C,\mathbf{x}}^*\nabla_{E_{\bC,\bx}f}\bV_{\by}^l}{h}_\mathcal{H_K} + 2\gamma_A\innerprod{f}{h}_{\cH_K} + 2\gamma_I\innerprod{S_{\bx}^*MS_{\bx}f}{h}_{\cH_K} \nonumber
    \\&=\innerprod{l^{-1}E_{C,\mathbf{x}}^*\nabla_{E_{\bC,\bx}f}\bV_{\by}^l + 2\gamma_Af +2\gamma_I{S_{\bx}^*MS_{\bx}f}}{h}_{\cH_K}, \quad h\in\cH_K. \label{e:dermin}
\end{align}
Keeping Fermat's Theorem \ref{t:fermat} in mind, since
that minimiser is an interior point of the domain, we should have the minimiser 
$f_{\bz, \gamma}\in \cH_K$ such that $\nabla_{f_{\bz,\gamma}} \cI (\cdot)=0$. By \eqref{e:dermin}, this means
\begin{equation}
    \nabla_{f_{\bz, \gamma}}\cI(h) = \innerprod{l^{-1}E_{C,\mathbf{x}}^*\nabla_{E_{\bC,\bx}{f_{\bz, \gamma}}}\bV_{\by}^l + 2\gamma_A{f_{\bz, \gamma}} +2\gamma_I{S_{\bx}^*MS_{\bx}{f_{\bz, \gamma}}}}{h}_{\cH_K} = 0
\end{equation}
for any $h\in \cH_K$. Thus, we get
\begin{equation*}
    \dfrac{1}{l}E_{C,\mathbf{x}}^*\nabla_{E_{\bC,\bx}{f_{\bz,\gamma}}}\bV_{\by}^l + 2\gamma_A{f_{\bz,\gamma}} +2\gamma_IS_{\bx}^*MS_{\bx}{f_{\bz,\gamma}} = 0
\end{equation*}
and hence
\begin{equation*}
    {f_{\bz,\gamma}} = \dfrac{-1}{2\gamma_Al}E_{C,\mathbf{x}}^*\nabla_{E_{\bC,\bx}{f_{\bz,\gamma}}}\bV_{\by}^l - \dfrac{\gamma_I}{\gamma_A}S_{\bx}^*MS_{\bx}{f_{\bz,\gamma}}.
\end{equation*}
Then by using the facts that 
$E_{\bC,\bx}^* = [K_{x_1}C_{x_1}^*\ \ K_{x_2}C_{x_2}^*\ \ \ldots\ \ K_{x_l}C_{x_l}^*]$, see \eqref{e:ecf}, 
and that, by definition, $S_{\bx}=[K_{x_1}^*\ \ K_{x_2}^*\ \ \ldots\ \ K_{x_{l+u}}^*]^T$, see \eqref{e:so},
we get
\begin{align*}
    {f_{\bz,\gamma}} &= \dfrac{-1}{2\gamma_Al}\sum_{i=1}^l K_{x_i}C_{x_i}^*(\nabla_{E_{\bC,\bx}{f_{\bz,\gamma}}}\bV_{\by}^l)_i - \dfrac{\gamma_I}{\gamma_A}\sum_{i=1}^{l+u}K_{x_i}(MS_{\bx}{f_{\bz,\gamma}})_i.\end{align*}
Thus, we can represent $f_{\bz,\gamma}$ as in \eqref{e:fezagam}, where
\begin{equation}
    a_i = \begin{cases}  \dfrac{-1}{2\gamma_Al}C_{x_i}^*(\nabla_{E_{\bC,\bx}{f_{\bz,\gamma}}}\bV_{\by}^l)_i - \dfrac{\gamma_I}{\gamma_A} (MS_{\bx}{f_{\bz,\gamma}})_i,& \mbox{ if }i=1,\dots,l\\ \\- \dfrac{\gamma_I}{\gamma_A} (MS_{\bx}{f_{\bz,\gamma}})_i,& \mbox{ otherwise.}
    \end{cases}
\end{equation}
Since
\begin{equation*}
    (MS_{\bx}f_{\bz,\gamma})_i = \sum_{k=1}^{l+u}M_{i,k}\sum_{j=1}^{l+u}K(x_k,x_j)a_j,\quad i=1,\ldots,l+u,
\end{equation*}
it follows that
\begin{equation*}
    a_i = \begin{cases}  \dfrac{-1}{2\gamma_Al}C_{x_i}^*(\nabla_{E_{\bC,\bx}{f_{\bz,\gamma}}}\bV_{\by}^l)_i - \dfrac{\gamma_I}{\gamma_A}\sum\limits_{j,k=1}^{l+u}M_{i,k}K(x_k,x_j)a_j,& \mbox{ if }i=1,\dots,l, \\ \\- \dfrac{\gamma_I}{\gamma_A}\sum\limits_{j,k=1}^{l+u}M_{i,k}K(x_k,x_j)a_j,& \mbox{ otherwise,}
    \end{cases}
\end{equation*}
which is equivalent to the system of equations \eqref{e:minval1} and \eqref{e:minval2}.
\end{proof}

\begin{remark}\label{r:opsol}
 {The system of equations \eqref{e:minval1} and \eqref{e:minval2} 
can be reformulated by means of operators; to be compared with
Theorem~4 in \cite{MBM} that was obtained for the special case of the least squares function.
Let $K[\mathbf{x}]$ denote the $(l+u)\times(l+u)$ operator valued matrix whose $(i,j)$ entry is $K(x_i,x_j)$, and
let $\mathbf{v}=(v_1,\ldots,v_{l+u})$ be the vector with entries
    \begin{equation*}
        {v}_i = \begin{cases}
            -C_{x_i}^*(\nabla_{E_{\bC,\bx}{f_{\bz,\gamma}}}\bV_{\by}^l)_i & \mbox{ if }i=1,\dots,l\\
            0,& \mbox{ otherwise,}
        \end{cases}
    \end{equation*}
    where  $C_x^*:\mathcal{Y}_x \to \mathcal{W}_x$ for any $x\in X$. Then, with notation and assumptions as in 
Theorem~\ref{t:gatmin} and letting $\mathbf{a} = (a_1,\dots,a_{l+u})$, a simple algebraic calculation, that we leave to 
the reader, shows that the system of equations \eqref{e:minval1} and \eqref{e:minval2}
coincides with the operator equation 
\begin{equation}\label{e:opmin}
  \left(2l\gamma_I M K[\mathbf{x}] + 2l\gamma_A I\right) \mathbf{a} = \mathbf{v},
\end{equation}
where both $\mathbf{a}$ and $\mathbf{v}$ are considered as column vectors. However, at this level of generality, 
from here we cannot get the unknown vector $\mathbf{a}$ because it appears in the vector $\mathbf{v}$ 
as well. From this point of view, the equation \eqref{e:opmin} is more an implicit form.}\end{remark}

\begin{corollary}\label{c:fdconvgat} With notation and assumptions as in Theorem~\ref{t:gatmin}, assume
that the loss functions $V_{x_i}$ are
convex for all $i=1,2,\ldots,l$. Then the minimisation problem \eqref{e:fezga} has
a unique solution $f_{\bz,\gamma}$ and there exist 
$a_1,\ldots,a_{l+u}$, with $a_j\in \cW_{x_j}$ for all $j=1,\ldots,l+u$, such that
\begin{equation*}
f_{\bz,\gamma}=\sum_{j=1}^{l+u} K_{x_j} a_j,
\end{equation*}
where the vectors $a_1,\ldots,a_{l+u}$ satisfy the system of equations \eqref{e:minval1} and \eqref{e:minval2}.
\end{corollary}

\begin{proof} The argument is the same as in Theorem~\ref{t:fdconv} and then use Theorem~\ref{t:gatmin}.
\end{proof}

Theorem~\ref{t:gatmin} and its Corollary~\ref{c:fdconvgat} have a 
theoretical significance and a less practical importance because, 
in the system \eqref{e:minval1} and \eqref{e:minval2}, 
the unknowns $a_1,\ldots,a_{l+u}$ appear also on the right hand side and hence it is 
more an implicit form of expressing them and not an explicit one. Because of that, for specified loss 
functions $V$, 
one should work further on these expressions in order to obtain explicit or, at least computable, solutions.
In the next two subsections, we work out the details and show how Theorem~\ref{t:gatmin} and its
Corollary~\ref{c:fdconvgat} can be improved for special loss functions, the least squares functions, similar 
to the results in \cite{MBM}, and the exponential least squares functions, and compare the formulae.

\subsection{The Least Squares Loss Function.}\label{ss:lslf}
If all the loss functions are the least squares, see Example~\ref{ex:lf}.(1), the minimisation function 
\eqref{e:ielef} becomes
\begin{align}\nonumber
\cI(f) & := \frac{1}{l}  \sum_{j=1}^l \|y_j-C_{x_j}f(x_j)\|^2_{\cY_{x_j}} 
+ \gamma_A \|f\|_{\cH_K}^2+ \gamma_I \langle \bef,M\bef\rangle_{\bW^{l+u}}\\
& = \frac{1}{l}  \sum_{j=1}^l \|y_j-C_{x_j}K_{x_j}^*f\|^2_{\cY_{x_j}} 
+ \gamma_A \|f\|_{\cH_K}^2+ \gamma_I \langle \bef,M\bef\rangle_{\bW^{l+u}},\label{e:ieles}
\end{align}
Since, as functions of $f\in\cH_{K}$, all the terms are convex and the middle term is strictly convex, see 
Lemma~\ref{l:semconv}, the minimisation
function $\cI(\cdot)$ is strictly convex and hence the minimisation problem has unique solution. Also, 
$\cI(\cdot)$ is Fr\'echet differentiable, hence Corollary~\ref{c:fdconvgat} is applicable. According to Fermat's 
Theorem, this unique solution $f$ should vanish the gradient. But, for each $h\in\cH_K$, we have
\begin{align*}
\nabla_f\cI(f)h = \frac{2}{l}\langle E_{\bC,\bx}^* E_{\bC,\bx}f,h\rangle_{\cH_K}-\frac{2}{l}\langle 
E_{\bC,\bx}^*\by,h\rangle_{\cH_K} + 2\gamma_A \langle f,h\rangle_{\cH_K}
+2\gamma_I \langle S_{\bx} M S_{\bx} \bef,h\rangle_{\cH_K},
\end{align*}
hence,
\begin{equation}\label{e:ebec}
E_{\bC,\bx}^* E_{\bC,\bx} f +l\gamma_A f+l\gamma_I S_{\bx}^*M S_{\bx}\bef-E_{\bC,\bx}^*\by=0.
\end{equation}
Since $\gamma_A>0$ this is equivalent with
\begin{equation*}
f=\frac{1}{l\gamma_A} E_{\bC,\bx}^* \by -\frac{1}{l\gamma_A} E_{\bC,\bx}^* E_{\bC,\bx}f-\frac{\gamma_I}{\gamma_A}
S_{\bx}^*MS_{\bx}\bef,
\end{equation*}
explicitly,
\begin{equation}
f  = \sum_{i=1}^l K_{x_i}\Bigl(\frac{1}{l\gamma_A}C_{x_i}^* y_i\Bigl) + \sum_{i=1}^l K_{x_i} 
\Bigl(-\frac{1}{l\gamma_A} C_{x_i}^* C_{x_i}f(x_i)\Bigl)  + \sum_{i=1}^{l+u} K_{x_i} \Bigl(-\frac{\gamma_I}{\gamma_A} 
\sum_{k=1}^{l+u} M_{i,k} f(x_k)\Bigr).\label{e:fesumas}
\end{equation}

In this special case, the representation \eqref{e:fesumas} improves
Theorem~\ref{t:gatmin} by obtaining the representation of the optimal solution 
as $f=\sum_{i=1}^{l+u} K_{x_i}a_i$, where, \begin{equation}\label{e:aifa}
a_i=\frac{1}{l\gamma_A} C_{x_i}^* y_i -\frac{1}{l\gamma_A} C_{x_i}^* C_{x_i} f(x_i)-\frac{\gamma_I}{\gamma_A}
\sum_{k=1}^{l+u} M_{i,k} f(x_k),\quad \mbox{ for all }i=1,\ldots,l,
\end{equation}
and, 
\begin{equation}\label{e:aifac}
a_i=-\frac{\gamma_I}{\gamma_A} \sum_{k=1}^{l+u} M_{i,k}f(x_k),\quad\mbox{ for all }i=l+1,\ldots,l+u.
\end{equation}
Then, since for all $k=1,\ldots,l+u$, we have
\begin{equation*}
f(x_k)=\sum_{j=1}^{l+u} K_{x_j}(x_k)a_j=\sum_{j=1}^{l+u}K(x_k,x_j)a_j,
\end{equation*}
and, consequently, from \eqref{e:aifa} and \eqref{e:aifac}, we get
\begin{align}\label{e:aisu}
a_i & + \sum_{j=1}^{l+u} \Bigl[ \frac{1}{l\gamma_A} C_{x_i}^*C_{x_i} K(x_i,x_j)+\frac{\gamma_I}{\gamma_A} \sum_{k=1}^{l+u} M_{i,k} K(x_k,x_j)\Bigl] a_j = \frac{1}{l\gamma_A} C_{x_i}^* y_i,\ i=1,\ldots,l,\\
a_i & + \sum_{j=1}^{l+u} \Bigl[ \frac{\gamma_I}{\gamma_A} \sum_{k=1}^{l+u} M_{i,k}K(x_k,x_j)\Bigr] a_j =0,\ i=l+1,\ldots, l+u.\label{e:aisum}
\end{align}

Equations \eqref{e:aisu} and \eqref{e:aisum} make a system of linear equations which can be treated very 
efficiently by computational techniques, similarly as in \cite{MBM} and the literature cited there. 
{This substantiates our claim that the localised versions offer more flexibility in modelling learning
problems without bringing additional obstructions.}

\subsection{The Exponential Least Squares Loss Function.}\label{ss:elslf}
If all the loss functions are the exponential least squares, see Example~\ref{ex:lf}.(4), the minimisation function 
\eqref{e:ielef} becomes
\begin{align}\nonumber
\cI(f) & := 1-\frac{1}{l}  \sum_{j=1}^l \exp(-\|y_j-C_{x_j}f(x_j)\|^2_{\cY_{x_j}}) 
+ \gamma_A \|f\|_{\cH_K}^2+ \gamma_I \langle \bef,M\bef\rangle_{\bW^{l+u}}\\
& = 1- \frac{1}{l}  \sum_{j=1}^l \exp(-\|y_j-C_{x_j}K_{x_j}^*f\|^2_{\cY_{x_j}}) 
+ \gamma_A \|f\|_{\cH_K}^2+ \gamma_I \langle \bef,M\bef\rangle_{\bW^{l+u}}.\label{e:ielefels}
\end{align}
In this case, since the second term is not, in general, convex, we cannot conclude that the function $\cI$ is 
convex. 
However, $\cI(f)\geq 0$ for all $f\in\cH_K$ and $\lim_{\|f\|\ra \infty}\cI(f)=+\infty$, hence the minimisation problem 
has at least one solution $f$, but this solution might not be unique. Anyhow, because any solution $f$ 
is an interior point in $\cH_K$, Fermat's Rule is applicable, hence $\nabla_f \cI=0$. Taking advantage of 
the calculations performed in the previous subsection, see \eqref{e:ebec}, 
by calculating the gradient of $\cI$ and then, by Fermat's 
Rule, the optimal function $f$ should satisfy the following equation
\begin{equation}\label{e:ebece}
\exp(-\|y_j-C_{x_j}f(x_j)\|^2_{\cY_{x_j}})\bigr(E_{\bC,\bx}^* E_{\bC,\bx} f -E_{\bC,\bx}^*\by\bigl)+l\gamma_A f
+l\gamma_I S_{\bx}^*M S_{\bx}\bef=0.
\end{equation}
Since $\gamma_A>0$ this is equivalent with
\begin{equation*}
f=\frac{\exp(-\|y_j-C_{x_j}f(x_j)\|^2_{\cY_{x_j}})}{l\gamma_A} \bigl(E_{\bC,\bx}^* \by - E_{\bC,\bx}^* E_{\bC,\bx}f\bigr)
-\frac{\gamma_I}{\gamma_A}S_{\bx}^*MS_{\bx}\bef,
\end{equation*}
explicitly,
\begin{align}
f  & = \sum_{i=1}^l K_{x_i}\Bigl(\frac{\exp(-\|y_j-C_{x_j}f(x_j)\|^2_{\cY_{x_j}})}{l\gamma_A}
\bigl(C_{x_i}^* y_i - C_{x_i}^* C_{x_i}f(x_i)\bigl)\Bigl)\nonumber \\
& \phantom{\sum_{i=1}^l K_{x_i}\Bigl(\frac{\exp(-\|y_j-C_{x_j}f(x_j)\|^2_{\cY_{x_j}})}{l\gamma_A}}+ \sum_{i=1}^{l+u} K_{x_i} \Bigl(-\frac{\gamma_I}{\gamma_A} \sum_{k=1}^{l+u} 
M_{i,k} f(x_k)\Bigr).\label{e:fesuma}
\end{align}

In this special case, \eqref{e:fesuma} improves Theorem~\ref{t:gatmin} by obtaining the representation of the optimal solution 
as $f=\sum_{i=1}^{l+u} K_{x_i}a_i$, where, 
\begin{equation}\label{e:aifas}
a_i=\frac{\exp(-\|y_j-C_{x_j}f(x_j)\|^2_{\cY_{x_j}})}{l\gamma_A}
\bigl(C_{x_i}^* y_i - C_{x_i}^* C_{x_i}f(x_i)\bigl),\quad \mbox{ for all }i=1,\ldots,l,
\end{equation}
and, 
\begin{equation}\label{e:aifacs}
a_i= -\frac{\gamma_I}{\gamma_A} \sum_{k=1}^{l+u} M_{i,k}f(x_k),
\quad\mbox{ for all }i=l+1,\ldots,l+u.
\end{equation}
Then, since for all $j=1,\ldots,l+u$, we have
\begin{equation*}
f(x_j)=\sum_{k=1}^{l+u} K_{x_k}(x_j)a_j=\sum_{k=1}^{l+u}K(x_j,x_k)a_k,
\end{equation*}
and, consequently, from \eqref{e:aifas} and \eqref{e:aifacs}, we get
\begin{align}\label{e:aisup}
a_i & + \sum_{j=1}^{l} \Bigl[ \frac{\exp(-\|y_j-C_{x_j}
\sum\limits_{k=1}^{l+u}K(x_j,x_k)a_k\|^2_{\cY_{x_j}})}{l\gamma_A} C_{x_i}^*C_{x_i} K(x_i,x_j)+\frac{\gamma_I}{\gamma_A} 
\sum_{k=1}^{l+u} M_{i,k} K(x_k,x_j)\Bigl]a_j \nonumber \\
& \hspace{5.6cm}  = \frac{1}{l\gamma_A} C_{x_i}^* y_i,\ i=1,\ldots,l,\\
a_i & + \frac{\gamma_I}{\gamma_A}\sum_{j=1}^{l+u}  \sum_{k=1}^{l+u} M_{i,k}K(x_k,x_j) a_j =0,\ i=l+1,\ldots, l+u.\label{e:aisumul}
\end{align}
Equations \eqref{e:aisup} and \eqref{e:aisumul} make a system of equations with respect to the unkowns
$a_i$ for $i=1,\ldots,l+u$, which consists of nonlinear equations for the unknowns $a_i$ 
corresponding to the labeled
input points and of linear equations for the unknowns $a_i$ corresponding to the unlabeled input points.

{
\subsection{Available Numerical Methods for the Exponential Least Square Loss Function.}\label{ss:aels}
In this subsection we tackle the question of deriving algorithms to solve the system of equations defined
by \eqref{e:aisup} and \eqref{e:aisumul}.
We use the same notations and assumptions as in the previous subsection. In addition, we assume that all 
Hilbert spaces $\cW_{x_i}$ have finite dimensions. To be more precise, $\cW_{x_i}$ is identified with 
$\RR^{d_i}$, for $i=1,\ldots,l+u$. We first define the vector
\begin{equation}\label{e:ba}
\ba=(a_1,a_2,\ldots,a_{l+u}) \in\RR^N,
\end{equation}
where $a_j\in \cW_{x_j}=\RR^{d_j}$, for each $j=1,\ldots,l+u$, and 
\begin{equation*}
N=\sum_{i=1}^{l+u} \dim(\cW_{x_i})=\sum_{i=1}^{l+u} d_i.
\end{equation*}
Then we consider the function $H\colon \RR^N\ra\RR^N$ with
\begin{equation*}
H(\ba)=(H_1(\ba),H_2(\ba),\ldots,H_{l+u}(\ba)),\quad \ba\in \RR^N,
\end{equation*}
defined by
\begin{align}
H_i(\ba) & = 
a_i  + \frac{1}{l\gamma_A}\sum_{j=1}^{l}  \exp\bigl(-\|y_j-C_{x_j}
\sum\limits_{k=1}^{l+u}K(x_j,x_k)a_k\|^2_{\cY_{x_j}}\bigr) C_{x_i}^*C_{x_i} K(x_i,x_j)a_j \nonumber\\
& \ \ \ \ +\frac{\gamma_I}{\gamma_A} \sum_{j=1}^{l+u}
\sum_{k=1}^{l+u} M_{i,k} K(x_k,x_j)a_j - \frac{1}{l\gamma_A} C_{x_i}^* y_i,\quad i=1,\ldots,l, \label{e:hiabal}\\
H_i(\ba) & = 
a_i  + \frac{\gamma_I}{\gamma_A}\sum_{j=1}^{l+u}   \sum_{k=1}^{l+u} M_{i,k}K(x_k,x_j) a_j,\quad i=l+1,\ldots, l+u.\label{e:hiabalu}
\end{align}
In view of the system of equations defined by \eqref{e:aisup} and \eqref{e:aisumul} with the unknown 
vector solution $\ba\in\RR^N$ as in \eqref{e:ba}, we search for solutions of the 
equation $H(\ba)=0$. From the numerical analysis point of view, 
this problem can be approached by nonlinear optimisation techniques, more precisely, we search for an 
algorithm that yields a sequence $(\ba_n)_{n\geq 0}$ of vectors in $\RR^N$ with the property that for each 
$\epsilon>0$ there exists an integer $n\geq 0$ such that $\|H(\ba_n)\|<\epsilon$. 
One of the classical approaches 
for this kind of nonlinear problems is in the class of Newton's damped approximation methods, see \cite{Hazely}  
for an overview and advances in second-order approximation methods and machine learning.

{There are different methods and algorithms for dealing with constrained nonlinear systems of differentiable functions. For example, the \emph{potential reduction 
Newton's method} in \cite{MonteiroPang} can be used in order to provide an algorithm to approximate
solutions $\ba\in\RR^N$ of the equation $H(\ba)=0$ under the assumption that the Jacobi matrix 
$\nabla_{\bu} H$ 
is nonsingular for all $\bu\in \RR^N$. Other methods based on the interior point methods are available, for 
example see \cite{Byrd99}, \cite{Byrd00}, and \cite{Waltz}. The latter algorithm is implemented in the 
\texttt{fmincon} function in MATLAB, that we have used in our example. }

{Because the optimisation problem is nonlinear, multiple solutions may show up and this makes 
the choice of the initialisation vector $\ba_0$ very important: for different choices of the initialisation vector 
different pools of solutions of the equations $H(\ba)=0$ might be found. Let us observe that, on the one hand, 
the components $H_i$ for $i=l+1,\ldots,l+u$, corresponding to unlabeled points $x_i$, 
are linear and homogeneous 
and hence that $0$ is a solution. On the other hand, the components $H_i$ for $i=1,\ldots,l$, corresponding to 
labeled points $x_i$, are nonlinear and nonhomogeneous and hence that $0$ is not a solution. From here, we 
can see that, on the one hand, as the 
learning problem is semisupervised, $l$ is significantly less than $u$ and hence taking the initialisation vector 
$\ba_0=0$ might be a good choice for the beginning. On the other hand, in order to find better minimisers, or 
even the global minimiser, some multigrid methods or stochastic methods 
for the choice of the initialisation vector $\ba_0$ might be involved, e.g.\ see \cite{Gower}.
In our example, we have used the \emph{Latin Hypercube Sampling} (LHS), see \cite{LHS_def}, 
to solve this issue. Similarly to multigrid methods, the LHS divides the domain into small cubes. Without loss of 
generality, assume that the domain is $[0,1]^d$ a cube with dimension $d$. Then for a fixed $n$, split the cube 
into $n^d$ subcubes by splitting each interval $[0,1]$ into intervals with $1/n$ length, that is,
\begin{equation*}
    [0,1] = \bigcup_{i=0}^{n-1} \left[\dfrac{i}{n},\dfrac{i+1}{n}\right].
\end{equation*}
Then we pick points such that any rectangle $R_{i,j}$ defined below contains only one point
\begin{equation*}
    R_{i,j} = [0,1]^{j-1}\times\left[\dfrac{i}{n},\dfrac{i+1}{n}\right]\times [0,1]^{d-j},\quad i,j =1,\dots,n.
\end{equation*}
This sampling strategy enjoys asymptotic bounds in expectation \cite{LHS_bound} with few points. Hence it is a 
good candidate for our purposes.}

Because there is no guarantee for uniqueness of solution, the algorithms for approximation of the solutions of the equation $H(\ba)=0$ should be complemented by further 
steps in which the obtained solutions should be tested whether they provide minimisers, of the learning 
function \eqref{e:fes}, or not and to which 
extent from the class of all minimisers one can get a global minimiser. But the most challenging problem refers to 
the assumption that the Jacobi matrix $\nabla_{\bu}H$ is nonsingular for all $\bu\in\RR^N$. In order to tackle this 
question we explicitly calculate this Jacobi matrix. To this end, we first observe that for each $i,j=1,\ldots,l+u$,
letting $a_j=(a_j^{(1)},\dots,a_j^{(d_j)})$,
we have a partial Jacobi matrix
\begin{equation*}
\frac{\partial H_i}{\partial a_j} = \Bigl(\frac{\partial H_i}{\partial a_j^{(1)}},\ldots, \frac{\partial H_i}{\partial 
a_j^{(d_j)}}\Bigr),
\end{equation*}
that is, a function matrix of dimension $d_i\times d_j=\dim(\cW_{x_i})\times \dim(\cW_{x_j})$. 
Let $i=1,\ldots,l$ and let $I_{d_i}$ denote the identity operator on $\cW_{x_i}$ identified with $\RR^{d_i}$. 
Then, from \eqref{e:hiabal}, on the one hand we get
\begin{align}
\frac{\partial H_i}{\partial a_i} & =  I_{d_i} + 
\frac{1}{l\gamma_A} \exp\bigl(-\|y_i-C_{x_j}
\sum\limits_{k=1}^{l+u}K(x_j,x_k)a_k\|^2_{\cY_{x_j}}\bigr) C_{x_i}^*C_{x_i} K(x_i,x_i) \nonumber\\
& + \frac{1}{l\gamma_A}\sum_{j=1}^{l+u}  2\langle y_j-C_{x_j} 
\sum\limits_{k=1}^{l+u}K(x_j,x_k)a_k, C_{x_j}K(x_j,x_i)a_i\rangle_{\cY_{x_i}}\times \nonumber\\
& \ \ \ \ \ \ \ \ \ \ \ \ \ \ \times
\exp\bigl(-\|y_j-C_{x_j}
\sum\limits_{k=1}^{l+u}K(x_j,x_k)a_k\|^2_{\cY_{x_j}}\bigr) C_{x_i}^*C_{x_i} K(x_i,x_j) \nonumber\\
& \ \ \ \ \ +\frac{\gamma_I}{\gamma_A} 
\sum_{k=1}^{l+u} M_{i,k} K(x_k,x_i),\label{e:parhii}
\end{align}
and, on the other hand, for each $j=1,\ldots,l+u$, $j\neq i$, we get
\begin{align}
\frac{\partial H_i}{\partial a_j} & = \frac{1}{l\gamma_A} \exp\bigl(-\|y_j-C_{x_j}
\sum\limits_{k=1}^{l+u}K(x_j,x_k)a_k\|^2_{\cY_{x_j}}\bigr) C_{x_i}^*C_{x_i} K(x_i,x_j) \nonumber \\
& \ \ \ \ + \frac{1}{l\gamma_A}\sum_{m=1}^{l+u}  2\langle y_m-C_{x_m} 
\sum\limits_{k=1}^{l+u}K(x_m,x_k)a_k, C_{x_m}K(x_m,x_j)a_j\rangle_{\cY_{x_m}}\times \nonumber \\
& \ \ \ \ \ \ \ \ \ \ \ \ \ \ \times
\exp\bigl(-\|y_m-C_{x_m}
\sum\limits_{k=1}^{l+u}K(x_m,x_k)a_k\|^2_{\cY_{x_m}}\bigr) C_{x_i}^*C_{x_i} K(x_i,x_m) \nonumber \\
& \ \ \ \ \ +\frac{\gamma_I}{\gamma_A} 
\sum_{k=1}^{l+u} M_{i,k} K(x_k,x_j).\label{e:parhij}
\end{align}
Let now $i=l+1,\ldots,l+u$. Then,  from \eqref{e:hiabalu}, on the one hand we get
\begin{align}
\frac{\partial H_i}{\partial a_i} & = I_{d_i} + \frac{\gamma_I}{\gamma_A} \sum_{k=1}^{l+u} M_{i,k} 
K(x_k,x_i),\label{e:parhiia}
\end{align}
and, on the other hand, for $j=1,\ldots,l+u$, $j\neq i$, we get
\begin{align}
\frac{\partial H_i}{\partial a_j} & =  \frac{\gamma_I}{\gamma_A} \sum_{k=1}^{l+u} M_{i,k} 
K(x_k,x_j).\label{e:parhija}
\end{align}

The partial Jacobi matrices obtained in \eqref{e:parhii} through \eqref{e:parhija}
make a complete description of the Jacobi matrix of the Fr\'echet derivative $\nabla_{\ba}H$. From these 
formulae, a few observations follow. A first observation is that we can write
\begin{equation}\label{e:nabah}
\nabla_{\ba}H = I_{\RR^N} + \frac{1}{\gamma_A}R,
\end{equation}
where $R$ is an $N\times N$ matrix that can be calculated explicitly from \eqref{e:parhii} through 
\eqref{e:parhija}. Then, one can use different extra assumptions on the kernels and 
data points in order to assure that $\nabla_{\ba}H$ is a nonsingular matrix. For example, one can take into 
account the fact that the labeled points 
$x_1,\ldots,x_l$ are selected in a supervised manner while the unlabeled points $x_{l+1},\ldots,x_{l+u}$ can be 
changed in a convenient manner that assures the Jacobi matrix $\nabla_{\ba}H$ be nonsingular. Another 
observation is that in a semisupervised learning problem the number $l$ of labeled points is significantly less 
than the number $u$ of unlabeled points and hence the degree of nonlinearity of the system of equations 
defined by \eqref{e:aisup} and \eqref{e:aisumul} is 
rather low, which can be used to search for reliable approximations  by systems of linear equations.

Some of the 
constrained optimisation problem described before depends heavily on the assumption that the Jacobi
matrix $\nabla_{\ba}H$ is nonsingular for all $\ba\in\RR^N$, a situation that may not be easy to get. In general, 
we can find a nonempty open set $\Omega$ in $\RR^N$ on which the Jacobi matrix $\nabla_{\ba}H$ is 
nonsingular, for example we can use \eqref{e:nabah} to show that $\Omega$ contains 
the set of those $\ba\in\RR^N$ with the 
property that $\|R\|<\gamma_A$. In view of  \eqref{e:parhii} through \eqref{e:parhija}, 
in the nonlinear terms of these partial Jacobi matrices the exponentials with negative exponents
tame the growth of $\|R\|$ when the vector $\ba$ is far from the solution and, consequently, 
by manipulating the regularisation coefficient $\gamma_A$ we can get very large sets $\Omega$.
Then, one can use other techniques to prevent the approximation sequence to get too close to 
the boundary of $\Omega$ as in  the constrained version of the algorithm proposed in 
\cite{MonteiroPang}. More recent investigations 
refer to either adapting the algorithm in case the Jacobi matrix $\nabla _{\bu}H$ is singular and pseudo-inverses 
replace the inverses, see e.g.\ the analyis of Kaczmarz type algorithms for ill-posed linear problems in 
\cite{Popa},  or using stochastic Bregman-Kaczmarz methods, see \cite{Gower}. These allow us to modify 
correspondingly the algorithm that we presented here to the general case. We leave the details for a further 
research project on real data sets. 
{As a practical approach, the simplest is to use perturbation theory, more 
precisely, for those values of $\ba$ at which the Fr\'echet derivative is not invertible, a small perturbation of $\ba$ 
changes the point to one for which the Fr\'echet derivative is invertible and then 
the stability of the problem to small 
perturbations guarantees the convergence of the iteration process.}

{
\subsection{A Toy Model.} In this subsection we provide a toy model for the localised version of the regularised 
machine learning  problem in case the loss function is the exponential least square function as in 
Subsection~\ref{ss:elslf} and test an algorithm following the discussion of the numerical methods as in Subsection~\ref{ss:aels}. To this end, let 
$X=X_1\cup X_2$, where
\begin{align}
X_1 & := \{(\alpha_1,\alpha_2)\mid 0.25\leq \alpha_1\leq 1, \ 0.25\leq \alpha_2\leq 1 \},\\
X_2 & := \{(\alpha_1,\alpha_2)\mid -1\leq \alpha_1\leq -0.25,\ 0.25\leq \alpha_2\leq 1\}.
\end{align}
In the following we use the notation as in Subsection~\ref{ss:grmlp}.
We consider $x_1\in X_1$ and $x_2\in X_2$ randomly selected and let the labels $y_1\in \cY_{x_1}=\RR$ and
$y_2\in \cY_{x_2}=\RR^2$ be randomly selected. Also, let $x_3,x_4\in X_1\setminus\{x_1\}$, $x_3\neq x_4$, 
and $x_5,x_6\in X_2\setminus\{x_2\}$, $x_5\neq x_6$, randomly 
selected as well, be unlabeled points. In particular, $l=2$ and $u=4$. Let $\cY_{x_3}=\cY_{x_4}=\RR$ and 
$\cY_{x_5}=\cY_{x_6}=\RR^2$. Then we take $\cW_{x_j}=\cY_{x_j}$ for all $j=1,\ldots,6$, in particular the 
machine learning problem is single viewed. With respect to the notation \eqref{e:welu} we have
\begin{align*}
\bW^{l+u} & = \cY_{x_1}\oplus  \cY_{x_2}\oplus  \cY_{x_3}\oplus  \cY_{x_4}\oplus  \cY_{x_5}\oplus  \cY_{x_6}\\
& = \RR\oplus \RR^2\oplus \RR\oplus \RR\oplus \RR^2\oplus \RR^2 = \RR^9,
\end{align*}
hence $\dim(\bW^{l+u})=9$.

We let the regularisation coefficients $\gamma_A$ and $\gamma_I$ unspecified and we will test different 
choices later. Since the problem is single-view we let 
$C_{x_i}=I_{\cY_{x_i}}$ for all $i=1,\ldots,6$ and $M_B=0$, see Example~\ref{ex:bevreg}. For the within-view
operator, see Example~\ref{ex:wivreg}, we proceed as follows. Let
\begin{equation}
w_{j,k} = \exp\bigl( - \frac{\|x_j-x_k\|^2}{2\sigma^2}\bigr),\quad j,k=1,\ldots,6,
\end{equation}
and then, consider the $6\times 6$ matrices $W=[w_{j,k}]_{j,k=1}^6$ and $V=\diag(v_{1,1},\ldots,v_{6,6})$, 
where
\begin{equation}
v_{j,j}= \sum_{k=1}^6 w_{j,k},\quad j=1,\ldots,6,
\end{equation}
and the Laplace matrix $L=V-W=[l_{j,k}]_{j,k=1}^6$. Taking into account that the labels of 
points in $X_1$ have dimension
$1$ and the labels of points in $X_2$ have dimension $2$ the matrix $M_W=M$ looks like this, 
see Example~\ref{ex:wivreg}.
\begin{equation}
M= \begin{bmatrix} l_{1,1} & l_{1,2} & 0 & l_{1,3} & l_{1,4} & l_{1,5} & 0 & l_{1,6} & 0 \\
l_{2,1} & l_{2,2} & 0 & l_{2,3} & l_{2,4} & l_{2,5} & 0 &  l_{2,6} & 0 \\
0 & 0 & l_{2,2} & 0 & 0 & 0 & 0 & 0 & 0 \\
l_{3,1} & l_{3,2} & 0 & l_{3,3} & l_{3,4} & l_{3,5} & 0 & l_{3,6} & 0 \\
l_{4,1} & l_{4,2} & 0 & l_{4,3} & l_{4,4} & l_{4,5}  & 0 & l_{4,6} & 0 \\
l_{5,1} & l_{5,2} & 0 & l_{5,3} & l_{5,4} & l_{5,5} & 0 & l_{5,6} & 0 \\
0 & 0 & 0 & 0 & 0 & 0 & l_{5,5} & 0 & 0 \\
l_{6,1} & l_{6,2} & 0 & l_{6,3} & l_{6,4} & l_{6,5} & 0 & l_{6,6} & 0 \\
0 & 0 & 0 & 0 & 0 & 0 & 0 & 0 & l_{6,6}
\end{bmatrix}.\label{e:matem}
\end{equation}
With notation as in Section~\ref{ss:grmlp} we have $M=[M_{j,k}]_{j,k=1}^6$ with 
\begin{equation*}
M_{1,1}=l_{1,1},\  M_{1,2}=\begin{bmatrix} l_{1,2} & 0\end{bmatrix},\ M_{1,3}=l_{1,3},\ M_{1,4}=l_{1,4},\ 
M_{1,5}=\begin{bmatrix} l_{1,5} & 0\end{bmatrix},\ M_{1,6}=\begin{bmatrix} l_{1,6} & 0 \end{bmatrix}, 
\end{equation*}
\begin{equation*}
M_{2,1} =\begin{bmatrix} l_{2,1} \\ 0 \end{bmatrix},\ M_{2,2} =\begin{bmatrix} l_{2,2} & 0 \\ 0 & l_{2,2} 
\end{bmatrix},\ M_{2,3}= \begin{bmatrix} l_{2,3} \\ 0 \end{bmatrix},\ M_{2,4} = \begin{bmatrix} l_{2,4} \\ 
0\end{bmatrix},\ M_{2,5} = \begin{bmatrix} l_{2,5} & 0 \\ 0 & 0\end{bmatrix},\ M_{2,6} = \begin{bmatrix}
l_{2,6} & 0 \\ 0 & 0 \end{bmatrix},
\end{equation*}
\begin{equation*}
M_{3,1}= l_{3,1}, \ M_{3,2} = \begin{bmatrix} l_{3,2} & 0 \end{bmatrix},\ M_{3,3}=l_{3,3},\ M_{3,4}=l_{3,4}, \
M_{3,5}= \begin{bmatrix} l_{3,5} & 0 \end{bmatrix},\ M_{3,6}= \begin{bmatrix} l_{3,6} & 0 \end{bmatrix},
\end{equation*}
\begin{equation*}
M_{4,1}= l_{4,1},\ M_{4,2} = \begin{bmatrix} l_{4,2} & 0 \end{bmatrix},\ M_{4,3}=l_{4,3},\ M_{4,4}=l_{4,4},\
M_{4,5}=\begin{bmatrix} l_{4,5} & 0 \end{bmatrix},\ M_{4,6}=\begin{bmatrix} l_{4,6} & 0\end{bmatrix},
\end{equation*}
\begin{equation*}
M_{5,1}= \begin{bmatrix} l_{5,1} \\ 0 \end{bmatrix},\ M_{5,2}=\begin{bmatrix} l_{5,2} & 0 \\ 0 & 0 \end{bmatrix},\
M_{5,3}=\begin{bmatrix} l_{5,3} \\ 0 \end{bmatrix},\ M_{5,4} =\begin{bmatrix} l_{5,4} \\ 0 \end{bmatrix} ,\  
M_{5,5}=\begin{bmatrix} l_{5,5} & 0 \\ 0 & l_{5,5}\end{bmatrix},\  M_{5,6}=\begin{bmatrix} l_{5,6} & 0 \\ 0 & 0 \end{bmatrix} , 
\end{equation*}
\begin{equation*}
M_{6,1}= \begin{bmatrix}l_{6,1} \\ 0 \end{bmatrix},\ M_{6,2}=\begin{bmatrix} l_{6,2} & 0 \\ 0 & 0\end{bmatrix},\
M_{6,3}=\begin{bmatrix} l_{6,3} \\ 0 \end{bmatrix} ,\
M_{6,4}=\begin{bmatrix} l_{6,4} \\ 0 \end{bmatrix} ,\
M_{6,5}=\begin{bmatrix} l_{6,5} & 0 \\ 0 & 0 \end{bmatrix} ,\
M_{6,6}= \begin{bmatrix} l_{6,6} & 0 \\ 0 & l_{6,6}\end{bmatrix}.
\end{equation*}

We consider the kernel $K\colon X\times X\ra \bigcup_{z,\zeta\in X} \cB(\cW_{\zeta},\cW_z)$ defined as follows.
\begin{align}
K(z,\zeta)& := \exp\bigl(-\frac{\|z-\zeta\|^2}{\sigma^2}\bigr),\quad z,\zeta\in X_1,\label{e:k11}\\
K(z,\zeta) & := \begin{bmatrix} \exp\bigl(-\frac{\|z-\zeta\|^2}{\sigma^2}\bigr) & 0 \\ 0 & \exp(-\alpha\|z-\zeta\|)
\end{bmatrix},\quad z,\zeta\in X_2,\label{e:k22}\\
K(z,\zeta) & :=  \begin{bmatrix} \exp\bigl(-\frac{\|z-\zeta\|^2}{\sigma^2}\bigr) & 0 \end{bmatrix},\quad z\in X_1,\ \zeta\in X_2,\label{e:k12}\\
K(z,\zeta) & := \begin{bmatrix} \exp\bigl(-\frac{\|z-\zeta\|^2}{\sigma^2}\bigr) \\ 0 \end{bmatrix},\quad z\in X_2,\ \zeta\in X_1. \label{e:k21}
\end{align}
Here the coefficients $\sigma$ and $\alpha$ remain unspecified for the moment.

With these data we have the system of equations \eqref{e:aisup} and \eqref{e:aisumul}, where $a_1$, $a_3$, 
and $a_4$ are scalars and $a_2$, $a_5$, and $a_6$ are $2$-vectors. So, speaking in terms of scalars, we have 
a system of nine equations with nine unknowns. Three of these equations are nonlinear and the rest of six 
equations are linear.

In order to find a bounded set on which we can guarantee the existence of the global solution of the 
minimisation problem \eqref{e:fezga}, we follow the idea of the proof of 
Theorem~\ref{t:fdcont}. With notation as in that theorem, we search for the minimiser
\begin{equation}\label{e:fes}
f=\sum_{j=1}^6 K_{x_j} a_j
\end{equation}
and we let $f_0=0$ and hence the corresponding vector $\ba=0$. Then, by \eqref{e:ielefels} we have
\begin{equation}\label{e:cefez}
\cI(f_0)= 1-\frac{1}{2}\bigl(\exp(-y_1^2)+\exp(-\|y_2\|^2)\bigr).
\end{equation}
We search for $\delta>0$ such that $\cI(f)\geq \cI(f_0)$, with $f$ as in \eqref{e:fes}, for any vector  
$\ba=(a_1,a_2,a_3,a_4,a_5,a_6)$ of dimension $9$ with the property that $\ba \not\in [-\delta,\delta]^9$.
By the proof of Theorem~\ref{t:fdcont},
the minimisation 
problem $\argmin\|H(\ba)\|$ has the solution in the cube $[-\delta,\delta]^9$. Taking into account that
\begin{equation}\label{e:cefeg}
\cI(f)\geq \gamma_A \|f\|_{\cH_K}\geq\gamma_A \langle K_{\bx}\ba,\ba\rangle \geq \gamma_A 
\lambda_{K_{\bx}} \|\ba\|_2^2\geq \gamma_A 
\lambda_{K_{\bx}} \|\ba\|_\infty^2,
\end{equation}
where, 
\begin{equation}K_{\bx}=[K(x_i,x_j)]_{i,j=1}^6,\label{e:kex}
\end{equation} 
is a $9\times 9$ matrix and 
$\lambda_{K_{\bx}}>0$ is the least eigenvalue of the positive matrix $K_{\bx}$. 
From \eqref{e:cefez} and
\eqref{e:cefeg}, letting
\begin{equation}\label{e:delta}
\delta=\frac{1}{\sqrt{\gamma_A \lambda_{K_{\bx}}}}
\sqrt{1-\frac{1}{2}\bigl(\exp(-y_1^2)+\exp(-\|y_2\|^2)\bigr)},
\end{equation}
it follows that in order to find the global minimiser for the problem \eqref{e:fezga}, it is sufficient to 
search for the solutions $\ba$ of
the minimisation problem $\argmin\|H(\ba)\|$ in the cube $[-\delta,\delta]^9$. 

There are two conditions to be verified, in order for $\delta$ to be consistent. Firstly, 
we work under the assumption that $\lambda_{K_{\bx}}>0$ which can 
be numerically checked, hence $\delta<\infty$. 
Secondly, note that if $\delta=0$ this means that $y_1=0$ and $y_2=0$, hence
$f=0$ is the solution for the global minimiser and hence, in this case, the problem is trivial. So, we work under 
the hypothesis that $\delta>0$. 

Since the optimisation problem is sensitive to the choice of initial conditions, we explored various possibilities. 
Specifically, we found that generating a \emph{Latin Hypercube Sampling} (LHS), see \cite{LHS_def} 
and the previous section, 
within the cube $[-\delta,\delta]^9$ provides a reliable estimate of the global solution while maintaining a 
reasonable runtime. Algorithm \ref{alg: num_imp} below summarises the numerical implementation.

\begin{remark}\label{r:adm}
 {In this algorithm, our admissibility criterion at line $7$ is $3$-fold. 
We check whether $\mathbf{a}_0$ results in both a 
smaller loss $\mathcal{I}(f_{\mathbf{a}_0})$, and a smaller gradient norm $\|\mathbf{H}(\mathbf{a}_0)\|$ 
compared to $\mathbf{a}$. Additionally, we verify that the first-order optimality condition is sufficiently small to 
ensure that $\mathbf{a}_0$ corresponds to a local extremum. Among all local optima we select the global one by 
a careful investigation and using the Latin Hypercube Sampling implemented in the algorithm and the code run 
on MATLAB. An essential part of the algorithm is the use of the function \texttt{fmincon} of MATLAB that uses the 
method of interior points for constrained optimisation, see \cite{Byrd99}, \cite{Byrd00}, and \cite{Waltz}.}
\end{remark}
\vfill\eject

\begin{algorithm}
\caption{Numerical Implementation}\label{alg: num_imp}
\begin{algorithmic}[1]
\REQUIRE $\mathbf{x}_i\in X_1 \cup X_2 \subset \mathbb{R}^2$, $\sigma,\alpha >0$, $\gamma_I,\gamma_A > 0$.
\STATE $D \gets \{\mathbf{x}_1, \mathbf{x}_2, \dots, \mathbf{x}_6\}$
\STATE Construct $M$ in \eqref{e:matem} and $K$ in \eqref{e:kex} using data $D$
\STATE Compute $\delta$ in \eqref{e:delta}
\STATE Compute the LHS $L$ on the cube $[-\delta, \delta]^9$
\FOR{each element $\mathbf{a}_0\in L$}
    \STATE Solve the minimisation problem \eqref{e:hiabal} and \eqref{e:hiabalu} on the cube $[-\delta, \delta]^9$ \\  with initial condition $\mathbf{a}_0$ using the function \texttt{fmincon}.
    \IF{$\mathbf{a}_0$ is admissible in the sense of Remark~\ref{r:adm}}
        \STATE $\mathbf{a}\gets \mathbf{a}_0$
    \ENDIF
\ENDFOR
\RETURN $\mathbf{a}$
\end{algorithmic}
\end{algorithm}

\begin{example}  {In this example we used $\gamma_A = 0.25$, $\gamma_I = 10$, $\sigma = 0.1$ 
and $\alpha = 10$, randomly generated the data $x$ and $y$ and returned the result $\ba$. }

\begin{align}
x & =
\begin{bmatrix}    0.5377  &  0.6342 &   0.3273 &   0.3472 &   0.6724 &   0.8174 \\
    0.3978 &  -0.4584 &   0.3923  &  0.4305  & -0.7962 &  -0.3601
\end{bmatrix} \ \ 
y =    (1.2108,    1.6636,    4.3843)\nonumber \\
\ba & =    (0.8433,    1.7226,    1.5475,    0.4395,    0.3944,    0.1926,   -0.0055,    1.4116,   -0.1589) \label{e:optima}
\end{align}

 {The implicit optimality tolerance  is $\epsilon=10^{-6}$. 
Each time the code produced the mesh of the solution $f_{best}$ given by \eqref{e:fes} for the corresponding 
coefficients given by the coordinates of the solution $\ba$. Below are the meshes of the optimal solution as in 
\eqref{e:optima}, the first one corresponds to the set $X_1$ where the function $f_{best}$ 
is scalar valued while 
the second one corresponds to the set $X_2$ where the function $f_{best}$ 
has $2$ dimensional vector values and 
there are two meshes, one for each component. The circles are the labeled points.}

\includegraphics[width=12cm, trim=0 245 0 235, clip]{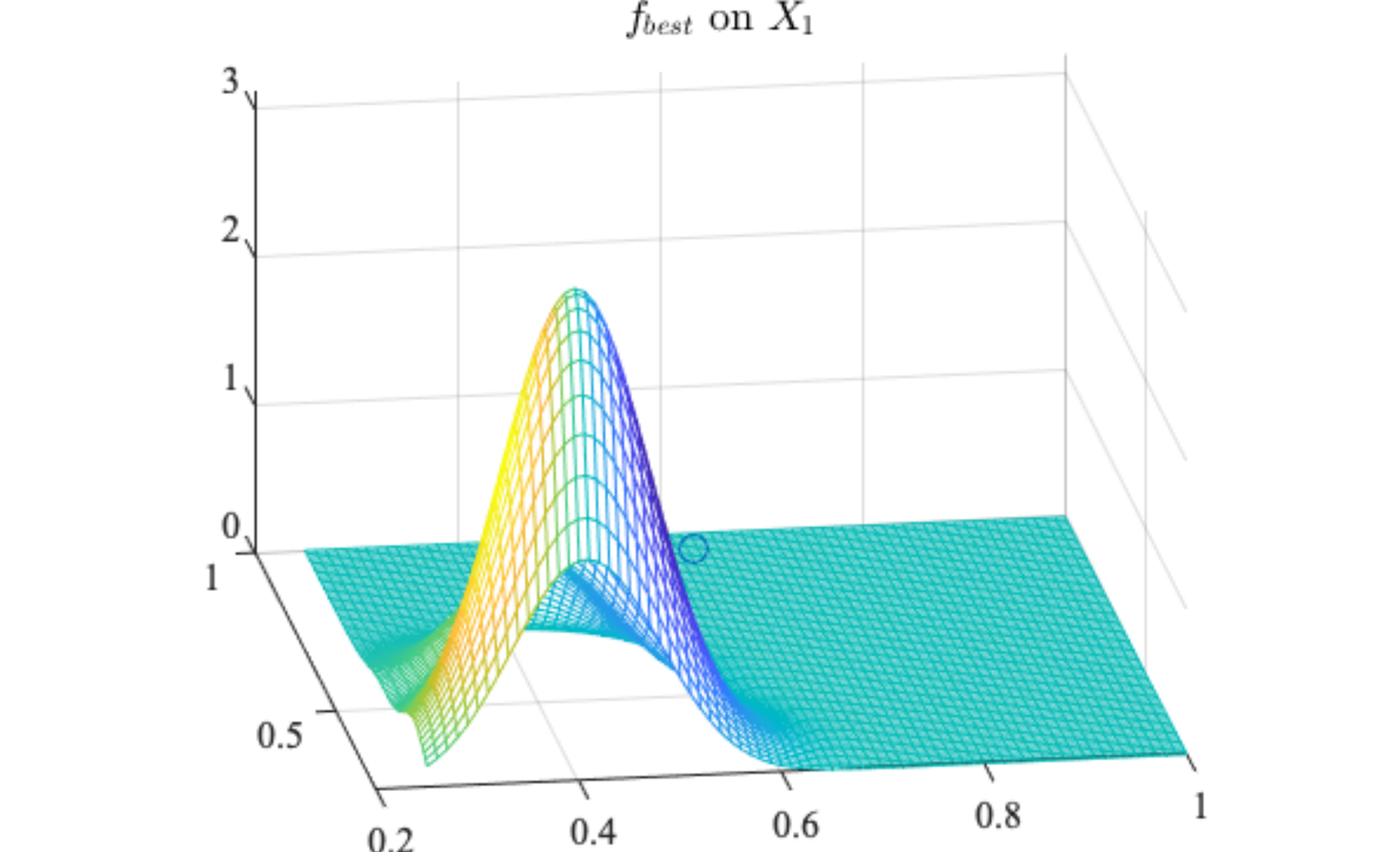}

\includegraphics[width=14cm, trim=0 235 0 235, clip]{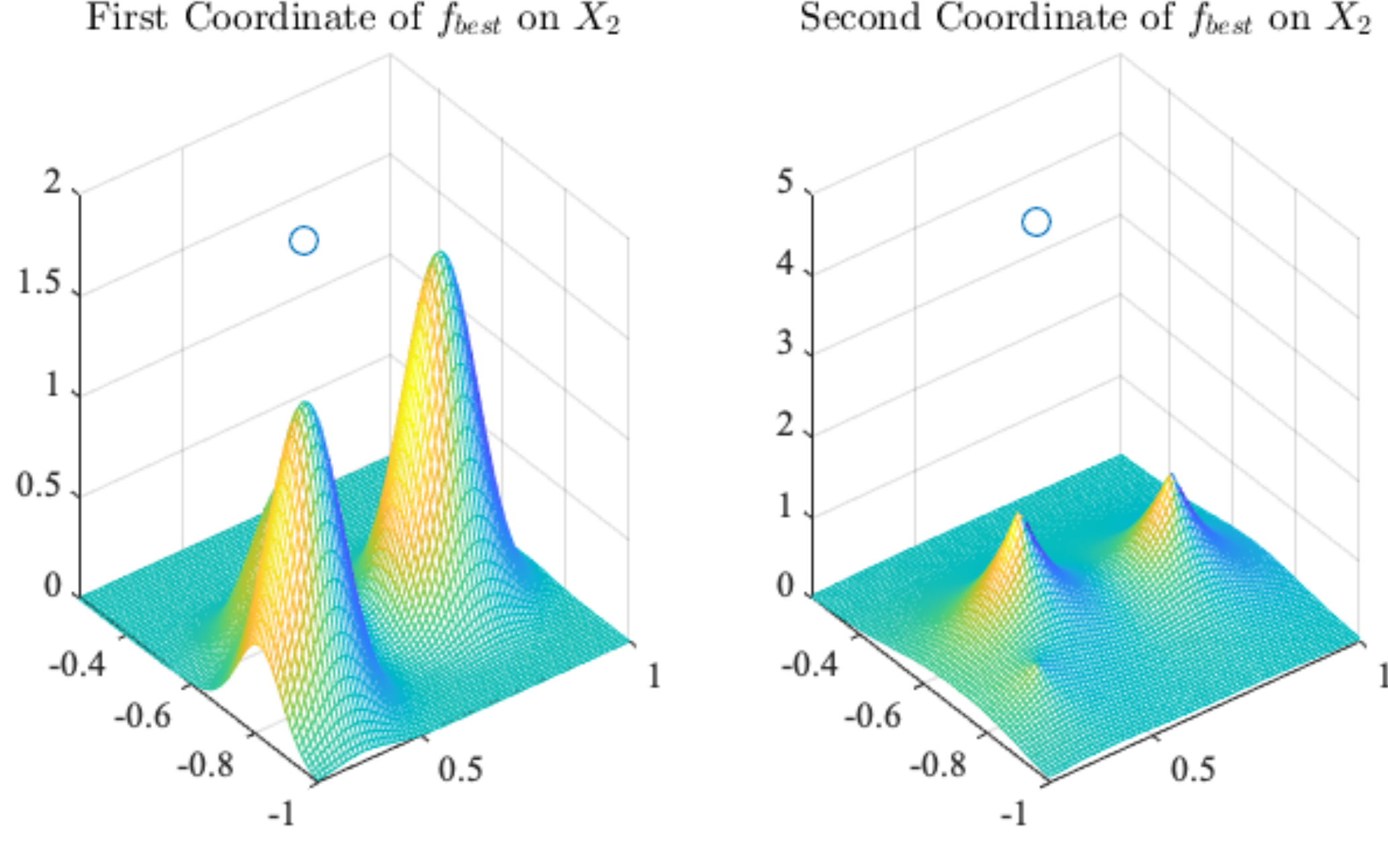}

 {We observe that the choice of the parameters avoids overfitting. Actually, by varying 
the coefficients $\gamma_A$, $\gamma_I$, and the others, one can obtain different degrees of overfitting. 
The code runs efficiently on a Mac laptop, for this example it takes only less than a minute, but for other 
combinations of coefficients it may be five or more minutes.}

 {We also plot the values of the learning function $\cI(f_{best})$ for each choice of the initial point}

\includegraphics[width=12cm, trim=0 235 0 235,clip]{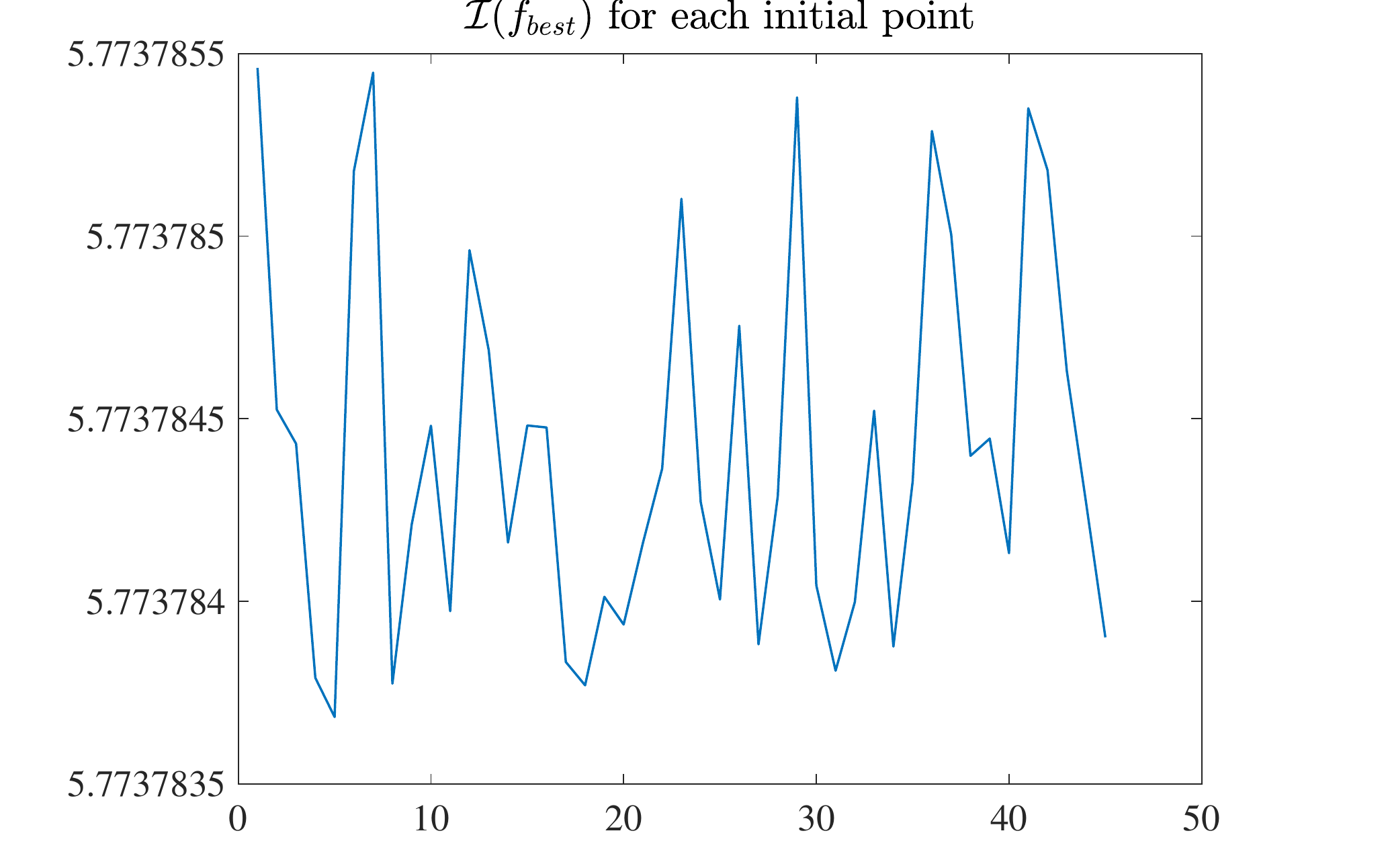}

\noindent  {and the values of $H(\ba)$, in order to empirically show that the optimal solution $\ba$ is a good 
approximation of the solution of the equation $H(\ba)=0$,  for each choice of the initial point,}

\includegraphics[width=12cm, trim=0 235 0 235,clip]{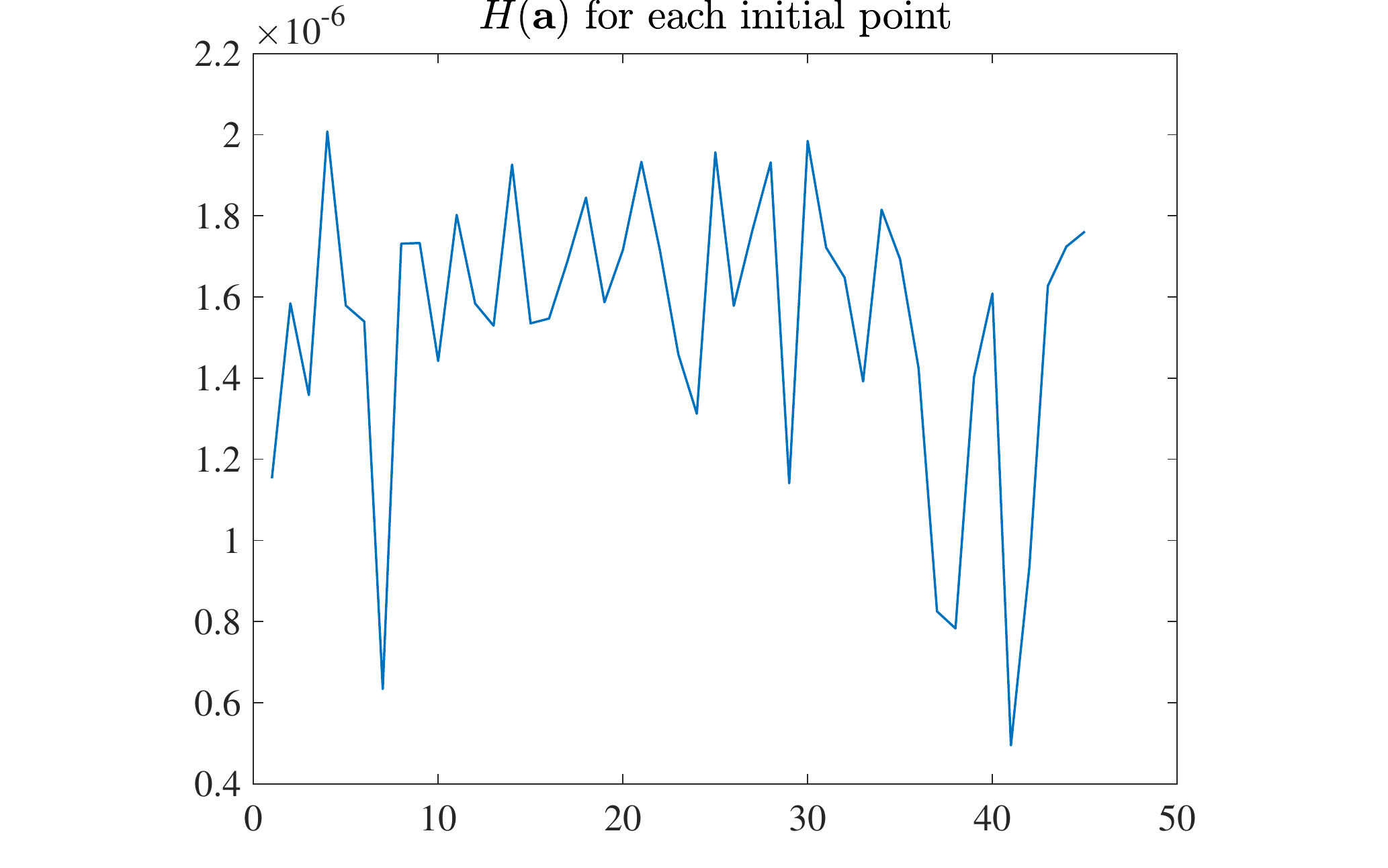}

 {In the last 
figure we empirically check that the solution $f_{best}$, corresponding to the optimal $\ba$ as in \eqref{e:fes}, 
is a good approximation of the minimiser of the learning function $\cI$, given at \eqref{e:ielefels}, for example 
by meshing it when we keep the last seven entries fixed as the last seven entries of the optimal $\ba$ 
and let each of 
the first two entries of $\ba$ vary in the interval $[-\delta,\delta]$. The point represents the value of the 
learning function at $f_{best}$.}

\includegraphics[width=12cm, trim=0 242 0 220,clip]{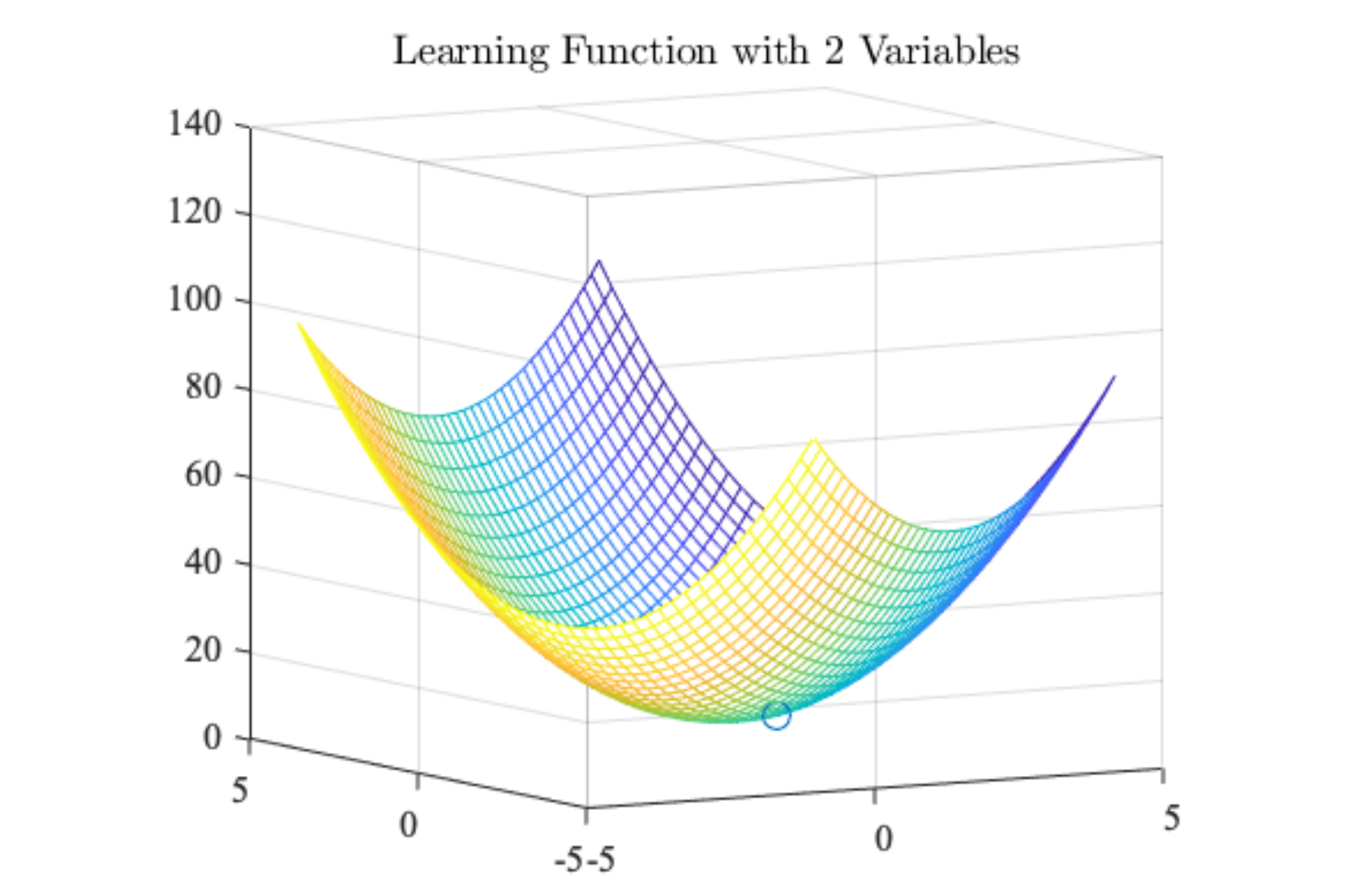}

 {Although the picture shows a convex surface this is misleading since it represents only a 
$2$-dimensional section of the general $9$-dimensional surface generated by the learning function $\cI(f)$, 
when $f$ is parametrised in terms of the $9$-dimensional vector $\ba$ as in \eqref{e:fes}. This behaviour of the 
learning map, as well as other traits, may vary for different choices of the coefficients $\gamma_A$, 
$\gamma_I$, $\sigma$, and $\alpha$, as we performed test for different combinations.}

 {In this example, one can observe that in the region that is trusted to contain 
the global minimum we actually 
have uniqueness of the solutions, both for the equation $H(\ba)=0$ and the minimiser of the learning map 
$\cI$. But, due to the nonlinearity of the system $H(\ba)=0$ and the nonconvexity of the learning map $\cI$, this 
does not mean that there may be no other solutions when the search is performed in a larger region, that is, 
local minima. This shows the importance of Theorem~\ref{t:fdcont} that provides a bounded region where the 
solutions that are of interest live.}

\end{example}
}



\appendix

\section{Proof of Theorem~\ref{t:kolmogorov}}
In the following we use a formalisation of the quotient completion to a Hilbert space of an $\FF$-vector space 
$\cV$ with respect to a given nonnegative sesquilinear form $\cV\times\cV\ni (u,v)\mapsto q(u,v)\in\FF$, as 
follows. A pair $(\cH;\Pi)$ is called a \emph{Hilbert space induced} by $(\cV;q)$ if:
\begin{itemize}
\item[(ihs1)] $\cH$ is a Hilbert space.
\item[(ihs2)] $\Pi\colon \cV\ra \cH$ is a linear operator with dense range.
\item[(ihs3)] $q(u,v)=\langle \Pi u,\Pi v\rangle_\cH$, for all $u,v\in\cV$.
\end{itemize} 
Such an induced Hilbert space always exists and is unique, up to a unitary operator. More precisely, we
will use the following construction. Consider the vector subspace of $\cV$ defined by 
\begin{equation}\label{e:nequ}
\cN_q:=\{u\in\cV\mid q(u,u)=0\}=\{u,\in\cV\mid q(u,v)=0\mbox{ for all }v\in\cV\},
\end{equation}
where the equality holds due to the Schwarz Inequality for $q$,
and then consider the quotient vector space $\cV/\cN_q$. Letting
\begin{equation}
\widetilde q(u+\cN_q,v+\cN_q):= q(u,v),\quad u,v\in \cV,
\end{equation}
we have a pre-Hilbert space $(\cV/\cN_q;\widetilde q)$ that can be completed to a Hilbert space 
$(\cH_q;\langle\cdot,\cdot\rangle_\cH)$. Letting $\Pi_q\colon \cV\ra \cH_q$ be defined by
\begin{equation}
\Pi_q u:=u+\cN_q\in \cV/\cN_q \subseteq \cH_q,\quad u\in \cV,
\end{equation} 
it is easy to see that $(\cH_q,\Pi_q)$ is a Hilbert space induced by $(\cV;q)$.\medskip

(a)$\Ra$(b). Assuming that the $\bH$-operator valued kernel $K$ is
positive semidefinite, we consider the vector space $\cF_0(X;\bH)$ of vector cross-sections
with finite support and the Hermitian sesquilinear form 
$\langle \cdot,\cdot\rangle_K$ defined as in \eqref{e:ipk}. We consider 
\begin{align}\label{e:cenek}
\cN_K & =\{f\in\cF_0(X;\bH)\mid \langle f,f\rangle_K=0\} \\
& =\{f\in\cF_0(X;\bH)\mid \langle f,g\rangle_K=0
\mbox{ for all }g\in F_0(X;\bH)\},\nonumber \end{align} then
consider the induced Hilbert space $(\cH_K;\Pi_K)$ associated to $(\cF_0(X;\bH);\langle\cdot,\cdot\rangle_K)$,
and let $\cK:=\cH_K$.
For each $x\in X$ let $V(x)\colon \cH_x\ra\cK$ be the operator defined by
\begin{equation}\label{e:vexeh}
V(x)h=\Pi_{K}(\widehat h)= h+\cN\in\cK,\end{equation}
with notation as in \eqref{e:deltax}. Since
\begin{equation} \langle V(x) h,V(x)h\rangle_K=
\langle K(x,x)h,h\rangle_{\cH_x}\leq \|K(x,x)\| \|h\|_{\cH_x}^2,
\quad h\in\cH_x,\ x\in X,
\end{equation} it follows that $V(x)$ is bounded for all $x\in X$. 
Note that, in this way, $\cK$ is the closed span of $\{V(x)\cH_x\mid x\in X\}$.
On the other hand,
\begin{equation} \langle K(y,x)h,g\rangle_{\cH_y}=\langle  h+\cN,g+\cN
\rangle_K=\langle V(x)h,V(y)g\rangle_K,
\quad h\in\cH_x,\ g\in\cH_y,\ x,y\in X,
\end{equation} hence, 
$K(y,x)=V(y)^*V(x)$ for all $x,y\in X$. 
We thus proved that $(\cK;V)$ is a minimal Hilbert space linearisation of the 
$\bH$-kernel $K$.

In the following we prove that $(\cK;V)$ is unique, modulo unitary equivalence. 
To see this, let $(\cK';V')$ be another minimal linearisation of $K$.
Then, for arbitrary $f\in\cF_0(X;\bH)$ we have
\begin{align*} \langle \sum_{x\in X} V(x) f_x,\sum_{y\in X}V(y) f_y\rangle_0 
& = \sum_{x,y\in X} \langle V(y)^*V(x) f_x,f_y\rangle_0 \\
& = \sum_{x,y\in X} \langle K(y,x) f_x,f_y\rangle_0 \\
& = \langle f,f\rangle_K \\
& = \langle \sum_{x\in X} V'(x) f_x,\sum_{y\in X}V'(y) f_y\rangle_0, 
\end{align*} hence, defining $U(\sum_{x\in X} V'(x) f_x)=\sum_{x\in X} V(x) f_x$, 
for arbitrary $f\in\cF_0(X;\bH)$, it follows that $U$ is isometric and, taking into 
account of the minimality conditions, it follows that $U$ can be uniquely extended 
to a unitary operator $U\colon \cK'\ra\cK$, such that $UV'(x)=V(x)$ for all $x\in X$.

(b)$\Ra$(a). Assuming that $(\cK;V)$ is a Hilbert space linearisation of $K$, we have
\begin{align*} \sum_{x,y\in X} \langle K(y,x)f_x,f_y\rangle_{\cH_y} 
& = \sum_{x,y\in X} \langle V(y)^*V(x)f_x,f_y\rangle_{\cH_y} \\
& = \|\sum_{x\in X} V(x) f_x\|^2_{\cK},\quad f\in\cF_0(X;\bH),
\end{align*} hence $K$ is positive semidefinite.

\section{Proof of Theorem~\ref{t:rkh}}
(a)$\Ra$(b). If $K$ is positive definite then, by 
Theorem~\ref{t:kolmogorov}, there exists a minimal linearisation 
$(\cK;V)$ of $K$. Define $\cR=\{V(\cdot)^* f\mid f\in\cK\}$, that is, $\cR$ consists of
all functions $X\ni x\mapsto V(x)^*f\in\cH_x$, with $f\in\cK$, in particular, 
$V(\cdot)^* f$ can be viewed as an $\bH$-vector bundle, that is, 
$V(\cdot)^*f\in\cF(X;\bH)$ for all $f\in\cK$. Thus, we can view $\cR$ as a linear 
subspace of $\cF(X;\bH)$, with all its algebraic operations.

We now show that the mapping
\begin{equation}\label{e:uf} \cK\ni f\mapsto Uf=V(\cdot)^*f\in\cR
\end{equation} is bijective. By definition, this mapping is surjective, hence it 
remains to prove that it is injective. To see this,  let $f,g\in \cK$ be such that 
$V(\cdot)^*f=V(\cdot)^* g$. Then for arbitrary $x\in X$ and $h\in\cH_x$ we have 
$\langle V(x)^*f,h\rangle_{\cH_x}=\langle V(x)^*g,h\rangle_{\cH_x}$, equivalently, 
$\langle f-g,V(x(h\rangle_\cK=0$. Taking into account the minimality of 
the linearisation, it follows that $f=g$. Thus, $U$ is bijective.

It is obvious  that the bijective mapping $U$ as in \eqref{e:uf} is linear. On $\cR$
we introduce an inner product $\langle\cdot,\cdot\rangle_\cR$ defined by
\begin{equation}\label{e:uip} \langle Uf,Uf\rangle_\cR=
\langle V(\cdot)^* f,V(\cdot)^* g\rangle_\cK,\quad f,g\in\cK,
\end{equation} in other words, $U$ is now an isometric isomorphism between the
Hilbert space $\cK$ and the inner product space $\cR$, hence 
$(\cR;\langle\cdot,\cdot\rangle_\cR)$ is a Hilbert space as well.

We now show that $(\cR;\langle\cdot,\cdot\rangle_\cR)$ is a reproducing kernel
Hilbert space with reproducing kernel $K$. Indeed, since for all $x,y\in X$ and all 
$h\in\cH_x$ we have $K_x(y)h=K(y,x)h=V(y)^*V(x)h$, it follows that $K_x\in\cR$ for
all $x\in X$. On the other hand, for arbitrary $f\in\cR$, $x\in X$, and $h\in\cH_x$,
we have
\begin{align*}\langle f,K_xh\rangle_\cR & = \langle V(\cdot)^*g,K_xh\rangle_\cR 
=\langle V(\cdot) g,V(\cdot)^*V(x)h\rangle_\cR \\
& =\langle g,V(x)h\rangle_\cK=\langle V(x)^*g,h\rangle_{\cH_x},
\end{align*} where $g\in\cH$ is the unique vector such that $V(x)^*g=f$. Thus, we
proved that $K$ is the reproducing kernel of $\cR$.

(b)$\Ra$(a). Let $(\cR;\langle\cdot,\cdot\rangle_\cR)$ be a 
reproducing kernel Hilbert space with reproducing kernel $K$. Using
the reproducing property (rk3), for arbitrary $n\in \NN$,
$x_1,\ldots, x_n\in X$, and $h_1\in \cH_{x_1},\ldots, h_n\in \cH_{x_n}$, 
we have
\begin{equation*} \sum_{i,j=1}^n \langle K(x_j,x_i)h_i,h_j
\rangle_{\cH_{x_j}} = \sum_{i,j=1}^n \langle K_{x_i}h_i,K_{x_j}h_j
\rangle_\cR = \|\sum_{i=1}^n K_{x_i}h_i\|_\cR^2\geq 0,
\end{equation*}
hence $K$ is positive semidefinite.

Due to the uniqueness property of the reproducing kernel Hilbert space associated 
to a positive semidefinite $\bH$-operator valued kernel $K$, it is natural to denote
this reproducing kernel Hilbert space by $\cR(K)$.

\section{A Direct Construction of $\cR(K)$.}
Given an arbitrary bundle of Hilbert 
spaces $\bH=\{\cH_x\}_{x\in X}$ and an $\bH$-operator valued kernel $K$,
we described the reproducing kernel Hilbert space $\cR(K)$ 
through a minimal linearisation of $K$, 
as in the proof of the implication 
(a)$\Ra$(b) of Theorem~\ref{t:rkh}, while a minimal Kolmogorov decomposition of 
$K$ was obtained as in the proof of the implication (a)$\Ra$(b) of 
Theorem~\ref{t:kolmogorov}. One of the unpleasant trait of the mentioned 
construction of the Kolmogorov decomposition, a GNS type construction in fact, 
is that, 
at a certain step, it makes a factorisation and hence, the obtained Hilbert space 
consists of equivalence classes of vector cross-sections. On the other hand, the reproducing 
kernel Hilbert space $\cH(K)$ consists solely of vector cross-sections and, as noted before,
it is a Kolmogorov decomposition as well,  hence it would be desirable to have a 
direct construction of it, independent of the Kolmogorov decomposition. Such a  
direct, but longer, 
construction, that yields simultaneously the reproducing kernel Hilbert space 
$\cR(K)$ and a minimal Kolmogorov decomposition of $K$, is more 
illuminating from certain points of view, and we describe 
it in the following.

Let $\cR_0$ be the range of the convolution operator $K$ defined at 
\eqref{e:conv}, more precisely, with the definition of the convolution operator $C_K$ as 
in \eqref{e:conv},
\begin{align}\label{e:rezero}
\cR_0 & =\{f\in\cF(X;\bH)\!\mid\! f=C_Kg\mbox{ for some }g\in\cF_0(X;\bH)\} \\
& = \{f\in\cF(X;\bH)\!\mid\! f_y=\sum_{x\in X}K(y,x)g_x\mbox{ for some }g\in
\cF_0(X;\bH),\mbox{ all }y\in X\}.\nonumber
\end{align}
A pairing $\langle\cdot,\cdot\rangle_{\cR_0}$ can be defined on $\cR_0$ by
\begin{equation}\label{e:iprezero} 
\langle e,f\rangle_{\cR_0}=\langle g,h\rangle_K=\langle C_Kg,h
\rangle_0=\sum_{y\in X} \langle e(y),h(y)\rangle_{\cH_y}=\sum_{x,y\in X}\langle 
K(y,x)g(x),h(y)\rangle_{\cH_y},
\end{equation} where $f=C_Kh$ and $e=C_Kg$ for some $g,h\in\cF_0(X;\bH)$.
We observe that, with the previous notation,
\begin{align}\label{e:pairzero} \langle e,f\rangle_{\cR_0} & = \sum_{y\in X} \langle e(y),h(y)
\rangle_{\cH_y} = \sum_{x,y\in X} \langle K(y,x)g(x),h(y)\rangle_{H_y} \\
& = \sum_{x,y\in X}\langle g(x),K(x,y)h(y)\rangle_{\cH_x} = \sum_{x\in X}\langle 
g(x),f(x)\rangle_{\cH_x},\nonumber
\end{align} which shows that the definition in \eqref{e:iprezero} is correct, that is,
it does not depend on $g$ and $h$ such that $e=C_Kg$ and $f=C_Kh$. 
In the following we prove that the pairing $\langle\cdot,\cdot\rangle_{\cR_0}$ is
an inner product. It is easy to verify the linearity in the first argument, conjugate 
symmetry, and nonnegativity. Hence, the Schwarz inequality holds as well.
In order to verify its positive definiteness, let $f\in\cR_0$ be 
such that $\langle f,f\rangle_{\cR_0}=0$. By the Schwarz inequality, it follows that 
$\langle f,f'\rangle_{\cR_0}=0$ for all $f'\in \cR_0$. For arbitrary $x\in X$ and $h\in 
\cH_x$ consider the cross-section $\widehat h\in \cF_0(X;\bH)$ defined as in 
\eqref{e:deltax}.
 Letting $f'=C_K\widehat  h\in
\cR_0$, we thus have
\begin{equation*} 0=\langle f,f'\rangle_{\cR_0}=\langle f,C_K\widehat  h\rangle_0=
\sum_{x\in X}\langle f_y,(\widehat  h)_y\rangle_{\cH_y}=\langle f_x,h\rangle_{\cH_x},
\end{equation*} hence, since $x\in X$ and $h\in\cH_x$ are arbitrary, it follows that 
$f=0$. Thus, 
$(\cR_0;\langle\cdot,\cdot\rangle_{\cR_0})$ is an  inner product space contained
in $\cF(X;\bH)$.

For any $x\in X$ and $h\in\cH_x$, we consider the vector cross-section 
$\widehat  h\in\cF_0(X;\bH)$ defined at \eqref{e:deltax} and note that
\begin{equation}\label{e:kax} (C_K\widehat  h)(y)=\sum_{z\in X} K(y,z)(\widehat  h)(z)
=K(y,x)h=K_x(y)h,\quad y\in X,\end{equation} 
that is, $C_K\widehat  h=K_x h$, which shows that $K_xh\in \cR_0$. 
On the other hand, for any $f\in\cR_0$, hence $f=C_Kg$ for some 
$g\in\cF_0(X;\bH)$, we have
\begin{equation} f(y)=\sum_{x\in X}K(y,x)g(x)=\sum_{x\in X}K_x(y)g(x),\quad y\in X,
\end{equation} hence $\cR_0=\lin\{K_xy\mid x\in X,\ h\in\cH_x\}$. In addition,
\begin{equation*} \langle f,K_xh\rangle_{\cR_0}  = \langle 
f,C_K\widehat  h\rangle_{\cR_0} 
 = \sum_{y\in X}\langle f(y),(\widehat  h)(y)\rangle_{\cH_y}=\langle f(x),h
\rangle_{\cH_x}.
\end{equation*}
Thus, the inner product space $(\cR_0;\langle\cdot,\cdot\rangle_{\cR_0})$ has
all  properties (rk1)--(rk3), as well as a modified version of the minimality 
property (rk4), except the fact that it is a Hilbert space.

By the standard procedure, let $(\cR;\langle\cdot,\cdot\rangle_{\cR})$ be an 
abstract completion of the inner product space 
$(\cR_0;\langle\cdot,\cdot\rangle_{\cR})$ to a Hilbert space. In order to finish this
construction, all we have to prove is that we can always choose $\cR\subseteq 
\cF(X;\bH)$, in other words, this Hilbert space abstract completion can always be 
realised inside $\cF(X;\bH)$. Once this done, after a moment of thought and taking 
into account that
$(\cR_0;\langle\cdot,\cdot\rangle_{\cR})$ essentially has all properties 
(rk1)--(rk4), we can see that $(\cR;\langle\cdot,\cdot\rangle_{\cR})$ is the 
reproducing kernel Hilbert space with reproducing kernel $K$.

Now, in order to prove that the Hilbert space abstract completion of 
$(\cR_0;\langle\cdot,\cdot\rangle_{\cR_0})$ can be realised within $\cF(X;\bH)$, we
can take at least two paths.
One way is to use the existence part of the reproducing kernel Hilbert 
space associated to $K$, a consequence of Theorem~\ref{t:rkh}. A second, more 
direct way, is to show that any Cauchy sequence, with respect to 
$\|\cdot\|_{\cR_0}$, with elements in $\cR_0$, converges pointwise on X to a 
vector cross-section in $\cF(X;\bH)$ and that this vector cross-section can be taken as 
the strong limit of the sequence as well. 

\section{Proof of Theorem~\ref{t:evop}} (a)$\Ra$(b). Let $x\in X$ be fixed, but arbitrary.
It was already observed in Subsection~\ref{ss:ovk}
that, if $\cH_K$ is the reproducing kernel Hilbert space in $\cF(X;\bH)$ with kernel $K$, 
then by the reproducing property, we have
\begin{equation*} \langle f(x),h\rangle_{\cH_x}=\langle f,K_xh\rangle_{\cH_K},\quad f\in \cH_K,\ h\in \cH_x,
\end{equation*}
where $K_x\colon \cH_x\ra \cH_K$ is the linear operator defined by $K_x h:=K(\cdot,x)h$, see the axiom (rk2). Since, 
by axiom (rk2), $K_x h\in \cH_K$ for all $h\in\cH_x$, the operator 
$K_x$ is correctly defined. It is a bounded operator because
\begin{equation}\label{e:kix} 
\|K_x h\|^2_{\cH_K}=\langle K_x h,K_x h\rangle_{\cH_K}=\langle (K_xh)(x),h\rangle_{\cH_x}
=\langle K(x,x)h\rangle_{\cH_x}\leq \|K(x,x)\| \|h\|_{\cH_x},
\end{equation}
where we have used the reproducing property (rk3). 

Finally, again by the reproducing property (rk3), for any $f\in\cH_K$ and any $h\in\cH_x$ we have
\begin{equation*}
\langle f(x),h\rangle_{\cH_x}=\langle f,K_x h\rangle_{\cH_K}=\langle K_x^* f,h\rangle_{\cH_x},
\end{equation*}
hence the evaluation operator $\cH_K\ni f\mapsto f(x)\in\cH_x$ coincides with $K_x^*$ and hence it is bounded.

(b)$\Ra$(a). For arbitrary $x\in X$, let $\mathrm{Ev}_x\colon \cH\ra\cH_x$ be the evaluation operator 
$\mathrm{Ev}_x f:=f(x)$, for all $f\in \cH$. By assumption, $\mathrm{Ev}_x$ is a bounded operator for all $x\in X$. 
We consider the $\bH$-valued kernel
\begin{equation*} K(y,x)=\mathrm{Ev}_y \mathrm{Ev}_x^*,\quad x,y\in X.
\end{equation*}
From Theorem~\ref{t:kolmogorov} it follows that $K$ is a positive semidefinite $\bH$-valued kernel and hence,
by Theorem~\ref{t:rkh}, there exists and it is unique, the reproducing kernel Hilbert space $\cH_K$ with kernel $K$.
In the following we show that $\cH$ is the reproducing kernel Hilbert space with kernel $K$.

The axiom (rk1) holds, by assumption. For the axiom (rk2), let us observe that, for all $x\in X$ and $h\in\cH_x$, we 
have 
\begin{equation*}(K_xh)(y)=K(y,x)h=\mathrm{Ev}_y \mathrm{Ev}_x^* h=(\mathrm{Ev}_x^*h)(y),\quad y\in X,
\end{equation*}
hence $K_xh=\mathrm{Ev}_x^*h\in \cH_K$. This proves that the axiom (rk2) holds and, in addition, that
\begin{equation*}
K_x^*=\mathrm{Ev}_x,\quad x\in X.
\end{equation*}
Finally, for the axiom (rk3), let $f\in \cH$, $x\in X$, and $h\in \cH_x$ be arbitrary. Then,
\begin{equation*}
\langle f(x),h\rangle_{H_x}=\langle \mathrm{Ev}_x f,h\rangle_{\cH_x}=\langle f,\mathrm{Ev}_x^* h\rangle_{\cH_K}
=\langle f,K_x h\rangle_{\cH_K}.
\end{equation*}
This shows that the axiom (rk3) holds as well.

Finally, by the uniqueness of the reproducing kernel Hilbert space associated to $K$, it follows that 
$\cH=\cH_K$.

\end{document}